\documentclass[11pt,reqno]{amsart}
\usepackage{setspace}
\usepackage[margin=1in]{geometry}
\usepackage[hang,flushmargin,symbol*]{footmisc}
\usepackage{amsmath}
\usepackage{amsthm}
\usepackage{amssymb}
\usepackage{mathtools}
\usepackage{enumitem}
\usepackage{calc}
\usepackage{graphicx}
\usepackage{caption}
\usepackage[labelformat=simple,labelfont={}]{subcaption}
\usepackage{tikz}
\usetikzlibrary{decorations.markings}
\usetikzlibrary{arrows,shapes,positioning}
\usepackage{url}
\usepackage{array}
\usepackage{graphicx}
\usepackage{color}
\usepackage{mathrsfs}

\usepackage{verbatim}

\usepackage{subfiles}

\usepackage[colorinlistoftodos]{todonotes}

\usepackage{dashbox}

\theoremstyle{definition}
\newtheorem{theorem}{Theorem}[section]
\newtheorem{lemma}[theorem]{Lemma}

\newtheorem{corollary}[theorem]{Corollary}
\newtheorem{proposition}[theorem]{Proposition}

\newtheorem{definition}[theorem]{Definition}
\newtheorem{example}[theorem]{Example}
\newtheorem{remark}[theorem]{Remark}

\newtheorem{construction}[theorem]{Construction}
\newtheorem{situation}[theorem]{Situation}

\usepackage{amsfonts}
\usepackage{amsmath}
\usepackage{amssymb}
\usepackage{mathrsfs}
\usepackage{tikz}
\usepackage{tikz-cd}
\usepackage{hyperref}
\usepackage{ mathdots }
\usepackage{dutchcal}
\usepackage{verbatim}
\usepackage{amsmath,amsthm}
\usepackage{extarrows}

\usepackage{xparse}

\usepackage{multirow}

\usepackage[normalem]{ulem}
\newcommand{\newstuffcolor}{black}

\newcommand{\msout}[1]{\text{\sout{\ensuremath{#1}}}}

\newcommand{\CC}{\mathbb{C}}
\newcommand{\NN}{{\mathbb{N}}}

\newcommand{\OO}{\mathcal{O}}

\newcommand{\N}[1]{{N_{#1}}}
\newcommand{\Nl}[1]{{N_{#1}^{\ell}}}

\newcommand{\Spec}{{\text{Spec}\:}}

\newcommand{\Hom}{\text{Hom}}

\newcommand{\EExt}{\underline{\text{Ext}}}

\newcommand{\ccx}[1]{{\mathbb{L}_{#1}}}
\newcommand{\lccx}[1]{{\mathbb{L}^\ell_{#1}}}

\newcommand{\vfc}[2]{ {[{#1}, {#2}]^{vir}} }
\newcommand{\lvfc}[2]{ {[{#1}, {#2}]^{\ell vir}} }

\newcommand{\glob}{\mathcal{V}}
\newcommand{\rectglob}{\mathcal{W}}

\newcommand{\Cl}[1]{{C_{#1}^\ell}}
\newcommand{\Clstrict}[1]{{C_{#1}^{\msout{\ell}}}}
\newcommand{\C}[1]{C_{#1}}

\newcommand{\Tl}[1]{{T^{\ell}_{#1}}}
\newcommand{\lkah}[1]{\Omega^\ell_{#1}}
\newcommand{\Log}{{\mathcal{L}}}

\newcommand{\lpb}{{\arrow[dr, phantom, very near start, "\ulcorner \ell"]}}
\newcommand{\lpbstrict}{{\arrow[dr, phantom, very near start, "\ulcorner \msout{\ell}"]}}
\newcommand{\pb}{{\arrow[dr, phantom, very near start, "\ulcorner"]}}

\newcommand{\Ms}{{\overline{M}_{g, n}}}
\newcommand{\Msi}[1]{{\overline{M}_{g, n} (#1)}}

\newcommand{\Msp}[1]{{\overline{M}_{g, n+#1}}}

\newcommand{\Ml}{{\mathscr{M}^\ell_{g, n}}}

\newcommand{\point}{{\overline{o}}}

\newcommand{\Mprel}{{\mathfrak{M}_{g, n}}}

\newcommand{\UU}{{\mathcal{U}}}
\newcommand{\UUU}{{\widetilde{\mathcal{U}}}}

\newcommand{\longequals}{\xlongequal{\: \:}}

\newcommand{\lpot}{{Log POT}\,}
\newcommand{\lvirt}{{Log VFC}\,}

\usetikzlibrary{decorations.markings}
\tikzset{smooth/.style={
        decoration={markings,
            mark= at position 0.5 with {
                \node[transform shape] (tempnode) {$\backslash$};
            }
        },
        postaction={decorate}
    }
}
\tikzset{etale/.style={
        decoration={markings,
            mark= at position 0.5 with {
                \node[transform shape] (tempnode) {$\backslash\backslash$};
            }
        },
        postaction={decorate}
    }
}

\newcommand{\onlyinsubfile}[1]{#1}
\newcommand{\notinsubfile}[1]{}

\newcommand{\widhat}[1]{\widehat{\hspace{0.1em}#1\hspace{0.15em}}}
\newcommand{\widtilde}[1]{\widetilde{\hspace{0.1em}#1\hspace{0.1em}}}

\title{\textbf{The Log Product Formula}}
\author[Leo Herr]{Leo Herr\\University of Utah \\ University of Colorado Boulder*}
\thanks{*The author's present address is the former. This article formed part of the author's dissertation at the latter.}
\date{\today}

\begin{document}

\maketitle

\renewcommand{\onlyinsubfile}[1]{}
\renewcommand{\notinsubfile}[1]{#1}

\setcounter{section}{-1}

\begin{comment}
Things to check before putting out final version: 

Fix all the bibliography citations: avoid automatic atrocious acronyms

widehat's okay

put noetherian, quasicompact, finite type in front of all stacks
\end{comment}

\begin{abstract}
    We prove a formula expressing the Log Gromov-Witten Invariants of a product of log smooth varieties $V \times W$ in terms of the invariants of $V$ and $W$. This extends results of \cite{logprodfmla}, which introduced this formula analogously to \cite{prodfmla}. The proof requires notions of ``log normal cone'' and ``log virtual fundamental class,'' as well as modified versions of standard intersection-theoretic machinery \cite{virtpb} adapted to Log Geometry. 
\end{abstract}

\section{Introduction}

\subsection{The Log Product Formula}

The purpose of the present paper is to prove the ``Product Formula'' for Log Gromov-Witten Invariants. For ordinary Gromov-Witten Invariants, the analogous formula was established by K. Behrend in \cite{prodfmla}. 

Let $V$, $W$ be log smooth, quasiprojective log schemes. Let $Q$ be the fs fiber product 
\[\Msi{V} \times^\ell_{\Ms} \Msi{W},\] 
with maps
\[\Msi{V \times W} \overset{h}{\longrightarrow} Q \overset{\widtilde{\Delta}}{ \longrightarrow } \Msi{V} \times \Msi{W}.\]
One can naturally endow $Q$ with a ``log virtual fundamental class'' in two ways: pushing forward that of $\Msi{V \times W}$ or pulling back that of $\Msi{V} \times \Msi{W}$. The Product Formula equates these: 

\begin{theorem}[The ``Log Gromov-Witten Product Formula'']
The two log virtual fundamental classes are equal in $A_*Q$:
\[h_* \lvfc{V \times W}{E(V \times W)} = \Delta^! \lvfc{V}{E(V)} \times \lvfc{W}{E(W)}.\]
\end{theorem}

The symbol $\Delta^!$ refers to a ``Log Gysin Map'' for which we offer a definition, along with ``Log Virtual Fundamental Classes.'' 

This theorem was formulated for ordinary virtual fundamental classes in \cite{logprodfmla} and proved under the assumption that one of $V$ or $W$ has trivial log structure. Like their work and \cite{prodfmla} before it, our proof centers on this cartesian diagram (Situation \ref{sit:thechaseddiagram}):
\[\begin{tikzcd}
\Ml(V \times W) \arrow[r, "h"] \arrow[d, "c"] \lpb        &Q \arrow[r] \arrow[d] \lpb     &\Ml(V) \times \Ml(W) \arrow[d, "a"]     \\
\mathfrak{D} \arrow[r, "l"]        &Q' \arrow[r, "\phi"] \arrow[d] \lpb       &\Mprel \times \Mprel \arrow[d, "s \times s"]      \\
        &\Ms \arrow[r, "\Delta"]        &\Ms \times \Ms.
\end{tikzcd}\]
One applies Costello's Formula \cite[Theorem 5.0.1]{costello} and commutativity of the Gysin Map to this diagram to compare virtual fundamental classes.

In the log setting, one requires this diagram to be cartesian in the 2-category of \textit{fs log algebraic stacks} in order to preserve modular interpretations. The assumption of \cite{logprodfmla} that $V$ or $W$ have trivial log structure ensures that these squares are \textit{also} cartesian as underlying algebraic stacks. 

\begin{comment}
The general theorem is false with ordinary virtual fundamental classes (Example \ref{ex:counterextoordinarylogprodfmla}) because 
\end{comment}

These fs pullback squares in question likely aren't cartesian on underlying algebraic stacks. Therefore, none of the standard machinery of ordinary Gysin Maps and Normal Cones is valid. This quandary forced us to prove the log analogues of Costello's Formula and commutativity for our ``Log Gysin Map.'' With these modifications, the original proof of K. Behrend essentially still works. We pause to comment on the new technology.

\subsection{Log Normal Cones}

The \textit{Log Normal Cone} $\Cl{X/Y} = \C{X/\Log Y}$ of a map $f : X \to Y$ of log algebraic stacks is the central object of the present paper. Every log map factors as the composition of a strict and an \'etale map $X \to \Log Y \to Y$, so the cone is determined by two properties: 
\begin{itemize}
    \item It agrees with the ordinary normal cone for strict maps.
    \item If one can factor $f$ as $X \to Y' \to Y$ with $Y' \to Y$ log \'etale, the cones are canonically isomorphic:
    \[\Cl{X/Y} \simeq \Cl{X/Y'}.\]
\end{itemize}

This object becomes simpler in the presence of charts. Locally, we may assume the map $X \to Y$ has a chart given by a map of Artin Cones $A_P \to A_Q$. The map $A_P \to A_Q$ is log \'etale, so we can base change across it to get a strict map without altering the log normal cone. 

Because this method can lead to radical alterations of the target $Y$, we recall another strategy that we learned from \cite[Proposition 2.3.12]{logmotives}. For ordinary schemes, one locally factors a map as a closed immersion composed with a smooth map to get a presentation for the normal cone \cite{intrinsic}. We obtain a similar local factorization (Construction \ref{katofactorization}) into a \textit{strict} closed immersion composed with a log smooth map, and the same presentation exists for the log normal cone. 

The above is made more precise in Remark \ref{axiomaticcharacterizationoflognormalcones}. The charts and factorizations these techniques require are only locally possible, so we need to know how log normal cones change after \'etale localization. We encounter a well-known subtlety noticed by W. Bauer \cite[\S 7]{logcotangent}: The log normal cone isn't invariant under base-changes by log \'etale maps (Remark \ref{rmk:stretalecaseoflciseschangesource}). Our workaround is somewhat different from that of Olsson. These results are at the service of log intersection theory, and we outline a standard package of log virtual fundamental classes and Log Gysin Maps.

\subsection{Pushforward and Gysin Pullback}

The proof of the Product Formula needs two ingredients: commutativity of Gysin maps and compatibility of pushforward with Gysin maps. The commutativity of Gysin Maps readily generalizes to the log setting in Theorem \ref{commlogpb}; on the other hand, compatibility with pushforward simply fails!

Nevertheless, the original proof of the product formula depends on a weak form of this compatibility first introduced by Costello \cite[Theorem 5.0.1]{costello}. We prove a log version of this theorem and will offer further complements in \cite{ourcorrectiontocostello}.

We obtain another partial result towards compatibility of pushforward and Gysin Pullback. For a log blowup $p : \widehat{X} \to X$ with a log smoothness assumption, we show $p_* [\widehat{X}]^{\ell vir} = [X]^{\ell vir}$ in Theorem \ref{pfwdoflvfcslblowup}. The alternative approach of \cite{logchowrecentpaper} may extend our results by modifying the notions of dimension, degree, pushforward, chow goups, etc. in the log setting. See also \cite{dhruvloggwthy} for an insightful approach to Log Chow Groups.

We hope the technology and the strategy of reducing statements about log normal cones to the strict, ordinary case will be of interest.

\subsection{Conventions}

\begin{itemize}
    \item We \textit{only consider fs log structures}. We therefore use $\Log, \Log Y$ to refer to Olsson's stacks $\mathscr{T}or, \mathscr{T}or Y$. 
    \item We work over the base field $\mathbb{C}$. 
    \item We adhere to the convention of \cite{logstacks} regarding the use of the term ``algebraic stack'': we mean a stack in the sense of \cite[3.1]{laumonmoretbaillystacks} such that
\begin{itemize}
    \item the diagonal is representable and of finite presentation, and 
    \item there exists a surjective, smooth morphism to it from a scheme. 
\end{itemize}
{\noindent}We do not require the diagonal morphism to be separated.
    \item By ``log algebraic stack,'' we mean an algebraic stack with a map to $\Log$. Maps between them need not lie over $\Log$.
    \item The name ``DM stack'' means Deligne-Mumford stack and a morphism $f : X \rightarrow Y$ of algebraic stacks is (of) ``DM-type'' or simply ``DM'' if every $Y$-scheme $T \rightarrow Y$ pulls back to a DM stack $T \times_{f, Y} X$ \cite{virtpb}.
    \item The word ``cone" in ``log normal cone'' refers to a cone stack in the sense of \cite{intrinsic}. 
    \item Let $P$ be a sharp fs monoid. Write 
    \[A_P = [\Spec \mathbb{C}[P]/\Spec \mathbb{C}[P^{gp}]]\]
    {\noindent}for the stack quotient in the \'etale topology endowed with its natural log structure \cite{wisebounded}, \cite{tropicalcurvesmodulicones}, \cite{logstacks}. Beware that some of these sources first take the dual monoid. This log stack has a notable functor of points for fs log schemes: 
    \[\Hom_{fs}(T, A_P) = \Hom_{mon}(P, \Gamma(\overline{M}_T)).\]
    In particular, 
    \[\Hom_{fs}(A_P, A_Q) = \Hom_{mon}(Q, P).\]
    We write $A$ for $A_{\mathbb{N}} = [\mathbb{A}^1/\mathbb{G}_m]$. Log algebraic stacks of this form are called ``Artin Cones.'' ``Artin Fans'' are log algebraic stacks which admit a strict \'etale cover by Artin Cones. The 2-category of Artin Fans is equivalent to a category of ``cone stacks'' \cite[Theorem 6.11]{tropicalcurvesmodulicones}. 
    
    \item The present paper concerns analogues of normal cones and pullbacks in the logarithmic category. We use the notation $\ulcorner$, $\times$, $\C{}$ for pullbacks and normal cones of ordinary stacks, and write $\ulcorner \ell$, $\times^\ell$, $\Cl{}$ to distinguish the fs pullbacks and log normal cones. When they happen to coincide, we write $\msout{\ell}$, $\ulcorner \msout{\ell}$, $\times^{\msout{\ell}}$, $\Clstrict{}$ to emphasize this coincidence. 
    
    \item Many of our citations could be made to original sources, often written by K. Kato, but we have opted for the book \cite{ogusloggeom}. We have doubled references to Costello's Formula \cite[Theorem 5.0.1]{costello}, \cite{ourcorrectiontocostello} where appropriate because we will have more to say building on future work. 
    
\end{itemize}

\ifdefined\isthesis{}
\else{
\subsection{Acknowledgments}

The present article is part of the author's Ph.D. thesis at the University of Colorado, Boulder under the supervision of Jonathan Wise. Not a result was envisioned, obtained, or fixed without his tremendous support, guidance, and patience. 

The author also benefitted from email correspondence with Dhruv Ranganathan and Lawrence Barrott. The author is grateful to the NSF for partial financial support from RTG grant \#1840190.
}
\fi

\section{Preliminaries and the Log Normal Sheaf}

The present paper originated with one central construction, which we learned from \cite[Lemma 2.3.12]{logmotives}. 

\begin{construction}\label{katofactorization}
The normal cone of a morphism $f : B \rightarrow A$ of finite type is constructed by choosing a factorization $B \rightarrow B[x_1, \dots, x_r] \twoheadrightarrow A$ inducing a closed immersion into affine $r$-space: 
\[\Spec A \hookrightarrow \mathbb{A}^r_B \rightarrow \Spec B.\]
The normal cone of $f$ may then be expressed as the quotient of the ordinary normal cone of the closed immersion by the action of the tangent bundle of $\mathbb{A}^r_B \rightarrow \Spec B$. 

Let $P \to A$ and $Q \to B$ be morphisms from fs monoids to the multiplicative monoids of rings (``prelog rings''). A commutative square:
\[\begin{tikzcd}
B \arrow[r, "f"]         &A      \\
Q \arrow[r, "\theta"] \arrow[u]       &P \arrow[u]
\end{tikzcd}\]
is a chart of a map between affine log schemes. Assume $f$ is of finite type; $\theta$ automatically is by the fs assumption. We will obtain a factorization of the induced log schemes into a \textit{strict closed immersion} followed by a \textit{log smooth map}. 

Start with a similar factorization
\[\begin{tikzcd}
B \arrow[r]       &B[x_1, \cdots, x_r, y_1, \dots, y_s] \arrow[r, two heads]       &A      \\
Q \arrow[r] \arrow[u]       &Q_s \arrow[r, two heads] \arrow[u]      &P \arrow[u]      \\
\end{tikzcd}\]
with $Q_s = Q \oplus \NN^s$ mapping to $B[x_1, \dots x_r, y_1, \dots y_s]$ by sending the generators of $\NN^s$ to the algebra generators $y_1, \dots y_s$. Define $Q_s^\theta$ via the cartesian product
\[\begin{tikzcd}
Q_s \arrow[r, hook] \arrow[d, hook]       &Q_s^\theta \arrow[r, two heads] \arrow[d, hook]  \arrow[dr, very near start, phantom, "\ulcorner"]      &P \arrow[d, hook]      \\
Q_s^{gp} \arrow[r, equals]       &Q_s^{gp} \arrow[r]        &P^{gp}.
\end{tikzcd}\]
By definition, $Q_s^\theta \rightarrow P$ is exact, and $Q_s \rightarrow Q_s^\theta$ is a ``log modification:'' an isomorphism on groupifications. Witness also that $Q_s^\theta \rightarrow P$ is surjective, so the characteristic monoid map $\overline{Q_s^\theta} \overset{\sim}{\rightarrow} \overline{P}$ is an isomorphism \cite[Proposition I.4.2.1(5)]{ogusloggeom} and $\Spec P \rightarrow \Spec Q_s^\theta$ is strict. Take $\Spec$ of both rings and monoids \cite[\S II]{ogusloggeom} to obtain a diagram with strict vertical arrows: 
\[\begin{tikzcd}
X \arrow[r, hook] \arrow[d]       &X_\theta \arrow[r] \arrow[d] \lpb       &\mathbb{A}^{r+s}_Y \arrow[r] \arrow[d]     &Y \arrow[d]      \\
\Spec P \arrow[r]     &\Spec Q^\theta_s \arrow[r]       &\Spec Q_s \arrow[r]        &\Spec Q
\end{tikzcd}\]

We've written $Y = \Spec B$, $X = \Spec A$ and introduced the fs pullback $X_\theta$ in the diagram. The top row expresses our original map $\Spec f$ as the composition of a strict closed immersion, a log modification, and a smooth and log smooth morphism. The log modification $\Spec Q^\theta_s \rightarrow \Spec Q_s$ and hence $X_\theta \rightarrow \mathbb{A}^{r+s}_Y$ may be expressed as a (strict) open immersion into a log blowup as in \cite[Lemma II.1.8.2, Remark II.1.8.5]{ogusloggeom}. Hence $X \subseteq X_\theta$ is a strict closed immersion and $X_\theta \rightarrow Y$ is log smooth. 

\end{construction}

\begin{remark}\label{katofactorizationwithsurjectivemonoids}
Continue in the notation of Construction \ref{katofactorization}. If we began with a morphism of fs log rings with $f$ and $\theta$ both surjective, we could omit $Q_s \rightarrow B[x_1, \dots, x_r, y_1, \dots y_s]$. In that case, we obtain a factorization 
\[X \subseteq X_\theta \rightarrow Y\]
where $X_\theta \rightarrow Y$ is not only log smooth but log \'etale. 
\end{remark}

As in \cite{intrinsic}, we will present the log normal cone locally as $\Cl{X/Y} = [\C{X/X_\theta}/\Tl{X_\theta/Y}]$ using these factorizations. The difficulty is then piecing together the local descriptions and checking compatibility. In this sense, the heavy lifting has already been done for us by \cite{virtpb}. We spend the rest of this section collecting relevant properties of the log normal sheaf $\Nl{X/Y}$. When we define the log normal cone $\Cl{X/Y} \subseteq \Nl{X/Y}$, its important properties will be locally deduced from such factorizations.

\begin{remark}
An algebraic stack $X$ is DM if and only if the map $X \rightarrow \Spec k$ to the base field is of DM-type. If $X \to Y$ is a morphism of DM type and $Y$ admits a stratification by global quotients, then so does $X$ \cite[Remark 3.2]{virtpb}. A morphism $f : X \to Y$ of algebraic stacks is of DM type if and only if its diagonal $\Delta_{X/Y} : X \to X \times_Y X$ is unramified \cite[06N3]{sta}. 
\end{remark}

\begin{lemma}
Let $f : X \rightarrow Y$ be a morphism of log algebraic stacks. If the map on underlying stacks is of DM-type, then the induced maps $\Log X \rightarrow \Log Y$ and $X \to \Log X$ are DM-type.
\end{lemma}

\begin{proof}

The inclusion $X \subseteq \Log X$ representing strict maps is open, so it suffices to show that $\Log X \to \Log Y$ is DM-type.

We will argue that the diagonal of $\Log X \to \Log Y$ is unramified \cite[04YW]{sta}. The isomorphism $\Log X \times_{\Log Y} \Log X \simeq \Log (X \times_Y^\ell X)$ identifies the diagonal $\Delta_{\Log X/\Log Y}$ with the result of $\Log$ applied to the fs diagonal 
\[\Delta^\ell_{X/Y} : X \to X \times_Y^\ell X.\]
Any diagram: 
\[\begin{tikzcd}
S_0 \ar[r] \ar[d, hook]         &\Log X \ar[d, "\Log \Delta^\ell_{X/Y}"]         \\
S_0' \ar[r] \ar[ur, dashed, shift left=1] \ar[ur, dashed, shift right=1]      &\Log (X \times_Y^\ell X)
\end{tikzcd}\]
with $S_0 \subseteq S_0'$ a squarezero closed immersion of schemes is equivalent to a diagram
\[\begin{tikzcd}
S \ar[r] \ar[d, hook]       &X \ar[d, "\Delta^\ell_{X/Y}"]      \\
S' \ar[r]\ar[ur, dashed, shift left=1] \ar[ur, dashed, shift right=1]      &X \times_Y^\ell X
\end{tikzcd}\]
with $S \subseteq S'$ an exact closed immersion of log schemes. Composing with the fsification map $X \times_Y^\ell X \to X \times_Y X$ sends this square to 
\[\begin{tikzcd}
S \ar[r] \ar[d, hook]       &X \ar[d, "\Delta^\ell_{X/Y}"]      \\
S' \ar[r]\ar[ur, dashed, shift left=1] \ar[ur, dashed, shift right=1]      &X \times_Y X,
\end{tikzcd}\]
in which case the two dashed arrows have the same underlying scheme map because $X \to X \times_Y X$ is unramified by hypothesis. Then the maps on log structure must be the same as well, because 
\[(M_X \oplus^\ell_{M_Y} M_X)|_{S'} \to (M_X)|_{S'}\]
is an epimorphism.

\end{proof}

Recall the functor of points of the normal sheaf. 

\begin{definition}[Normal Sheaf Functor of Points]\label{ordinarynormalsheaffunctorofpoints}
Let $f: X \rightarrow Y$ be a DM morphism of algebraic stacks. Define a stack $\N{X/Y}$ over $X$ named the \textit{log normal sheaf} via its functor of points: 
\[\left\{
\begin{tikzcd}
        &\N{X/Y} \arrow[d]        \\
T \arrow[r] \arrow[ur, dashed]       &X
\end{tikzcd}
\right\} := 
\left\{
\begin{tikzcd}
(T, \OO_X|_T) \arrow[r] \arrow[d, hook, "i"]       &X \arrow[d]      \\
(T, \mathcal{A}) \arrow[r]        &Y
\end{tikzcd} \begin{tikzcd}[row sep=tiny]
i \text{ is a square-zero closed}     \\
\text{immersion with kernel }\OO_T
\end{tikzcd}
\right\}\]
\[=\left\{\begin{tikzcd}
        &       &\OO_Y|_T \arrow[d] \arrow[dr, bend left]      &       &       &\text{a squarezero algebra}   \\
0 \arrow[r]       &\OO_T \arrow[r]      &\mathcal{A} \arrow[r]        &\OO_X|_T \arrow[r]       &0      &\text{extension on \'et$(T)$}
\end{tikzcd}\right\}\]

An obstruction theory for $f$ is a fully faithful functor $\N{X/Y} \subseteq E$ into a vector bundle stack as in \cite[Corollary 3.8]{obstthies}. 

\end{definition}

The notion of ``square-zero closed immersion'' in the definition demands elaboration, since the objects involved are \'etale-locally ringed spaces. See \cite{onnormalsheaves} for details.

\begin{remark}
Suppose we specified an obstruction theory $E_\bullet \rightarrow \ccx{X/Y}$ in the sense of \cite{intrinsic}. The associated obstruction theory according to Definition \ref{ordinarynormalsheaffunctorofpoints} on $T$-points is given by:
\[\N{X/Y}(T) = \EExt(\ccx{X/Y}|_T, \OO_T) \longrightarrow E = \EExt(E_\bullet|_T, \OO_T).\]
See \cite[Corollary 4.9]{obstthies} and \cite[Chapitre VIII: Biextension de faisceaux de groupes]{sga7-1} for comparison and elaboration on $\EExt(E_\bullet, J) = \Psi_{E_\bullet}(J)$. In particular, our obstruction theories are all representable by obstruction theories in the sense of \cite{intrinsic}. 
\end{remark}

\begin{definition}[The Log Normal Sheaf]\label{defn:lognormalsheaf}
Let $f : X \to Y$ be a DM morphism of log algebraic stacks. Let $T \to X$ be an $X$-scheme. A \textit{deformation of log structures along $f$ on $T$} is a log structure $M_\mathcal{A} \to \mathcal{A}$ on the \'etale site $\text{\'et}(T)$ of $T$ with maps $(\OO_Y|_T, M_Y|_T) \to (\mathcal{A}, M_\mathcal{A}) \to (\OO_X|_T, M_X|_T)$ of log structures such that:
\begin{itemize}
    \item The kernel $\ker (\mathcal{A} \to \OO_X|_T) \simeq \OO_T$ and the diagram
    \[\begin{tikzcd}
        &       & \OO_Y|_T \arrow[d] \arrow[dr, bend left]      &       &       \\
0 \arrow[r]       &\OO_T \arrow[r]     &\mathcal{A} \arrow[r]       &\OO_X|_T \arrow[r]     &0       
    \end{tikzcd}\]
    constitutes a squarezero algebra extension. 
    \item The diagram
    \[\begin{tikzcd}
    \mathcal{A}^* \arrow[r] \arrow[d]       &\OO_X^*|_T \arrow[d]         \\
    M_\mathcal{A} \arrow[r]       &M_X|_T
    \end{tikzcd}\]
    is a pushout. 
\end{itemize}

The second bullet says that $(\mathcal{A}, M_\mathcal{A})$ is a \textit{strict} squarezero extension of $(\OO_X|_T, M_X|_T)$; compare with ``deformations of log structures'' \cite{illusielog}. The square in the second bullet is also a pullback, and $M_\mathcal{A} \to M_X|_T$ is also a torsor under $1 + \OO_T$. 

Define the \textit{log normal sheaf} to represent the deformations of log structures just defined: 
\[\left\{
\begin{tikzcd}
        &\Nl{X/Y} \arrow[d]        \\
T \arrow[r] \arrow[ur, dashed]       &X
\end{tikzcd}
\right\} := \{\text{Deformations of log structures along $f$ on $T$}\}.\]

\end{definition}

We show that this definition agrees with Definition \ref{ordinarynormalsheaffunctorofpoints} in \cite{onnormalsheaves}: $\Nl{X/Y} = \N{X/\Log Y}$.

To write down the functoriality of the log normal sheaf, we need to recall some of the machinery of log stacks found in \cite{logcotangent}. 

We denote $\Log^i := \Log^{[i]}$, the stack of $i$-simplices of fs log structures. The $j$th face map $d_j$ sends
\[(M_0 \to M_1 \to \cdots \to M_{i+1}) \mapsto \left\{\begin{tikzcd}
(M_1 \to M_2 \to \cdots \to M_{i+1})      &\text{if }j = 0\\
(M_0 \to \cdots M_{j-1} \to M_{j+1} \cdots \to M_{i+1})   &\text{if }j \neq 0, i+1       \\
(M_0 \to \cdots \to M_{i})      &\text{if }j = i+1.
\end{tikzcd}\right.\]
We write $s, t : \Log^1 \to \Log^0 = \Log$ for the ``source'' $d_1$ and ``target'' $d_0$ maps, respectively. We have an isomorphism $\Log^i = \Log^1 \times_{t, \Log, s} \Log^1 \times_{t, \Log, s} \cdots \times_{t, \Log, s}\Log^1$ ($i$ factors). 

Endow $\Log^i$ with the final tautological log structure, $M_{i+1}$ in the above. All the face maps $d_j$ are strict except $j = i+1$. 

We continue \cite{logcotangent} to use ``$\square$'' to denote the category with these objects, arrows, and relations: 
\[\begin{tikzcd}
0 \arrow[r] \arrow[d] \arrow[dr] \arrow[dr, phantom, bend left, "\circ"] \arrow[dr, phantom, bend right, "\circ"]       &1 \arrow[d]      \\
2 \arrow[r]       &3
\end{tikzcd}\]
We adopt pictorial mnemonics for fully faithful morphisms of these finite diagrams: $\begin{tikzpicture}[scale = 0.4]
\draw[dotted] (0, 0) -- (0, 1);
\draw[dotted] (0, 0) -- (1, 0);
\draw (0, 1) -- (1, 0);
\draw (1, 0) -- (1, 1);
\draw (0, 1) -- (1, 1);
\end{tikzpicture}$ means the functor $[2] \subseteq \square$ avoiding 2, etc.

\begin{definition}[{Compare \cite[Lemma 3.12]{logcotangent}}]\label{cocartesianmonoidsquares}
Define $\glob := \Log^1 \times_{t, \Log, t}^\ell \Log^1$. Given a scheme $T$, the points of this stack are cocartesian squares of \textit{fs} log structures:
\[\glob(T) := \left\{\begin{tikzcd}
M_0 \arrow[r] \arrow[d] \arrow[dr, phantom, very near end, "\ell \lrcorner"]         &M_1 \arrow[d]        \\
M_2 \arrow[r]         &M_3.
\end{tikzcd}\right\}\]
This is the ``fsification'' of the ordinary pullback $\Log^1 \times_{t, \Log, t} \Log^1$, endowed with the non-fs pushout $M_1 \oplus_{M_0}^{mon} M_2$ of the universal log structures. 

The natural embedding $\glob \to \Log^\square$ exhibits the squares which are cocartesian as an open substack, as we'll record in Lemma \ref{olssoncotcplxfsversion}. 

For a morphism $q : Y' \to Y$ of log algebraic stacks, we obtain relative variants:
\[\glob_q := \glob \times_{\begin{tikzpicture}[scale = 0.2]
    \draw[dotted] (0, 0) -- (0, 1);
    \draw[dotted] (0, 0) -- (1, 0);
    \draw[dotted] (1, 0) -- (1, 1);
    \draw (0, 1) -- (1, 1);
    \end{tikzpicture}, \:\: \Log^1} Y',       \quad \quad \quad       \Log^\square_q := \Log^\square \times_{\begin{tikzpicture}[scale = 0.2]
    \draw[dotted] (0, 0) -- (0, 1);
    \draw[dotted] (0, 0) -- (1, 0);
    \draw[dotted] (1, 0) -- (1, 1);
    \draw (0, 1) -- (1, 1);
    \end{tikzpicture}, \:\: \Log^1} Y'.\]
The fs pullback here agrees with the ordinary one because $Y' \to \Log^1$ is strict. The points of these stacks over some scheme $T$ are squares
\[\begin{tikzcd}
M_Y|_T \arrow[r] \arrow[d]     &M_{Y'}|_T \arrow[d]      \\
M_0 \arrow[r]     &M_1,
\end{tikzcd}\]
with those of $\glob_q$ required to be cocartesian. 

\end{definition}

\begin{lemma}\label{fsisopenimmersioninfinelstrs}
Let $\Log^{arb fine}$ denote the stack of log structures which are fine but not necessarily saturated. The natural monomorphism 
\[\Log \hookrightarrow \Log^{arb fine}\]
is an open immersion. 
\end{lemma}

\begin{proof}

Consider some scheme $X$ and pullback diagram
\[\begin{tikzcd}
X^{fs} \arrow[r] \arrow[d] \arrow[dr, phantom, very near start, "\ulcorner"]      &\Log  \arrow[d]     \\
X \arrow[r]       &\Log^{arbfine}
\end{tikzcd}\]

Then $X^{fs} \hookrightarrow X$ is a monomorphism, the locus where the stalks of $M_X$ are saturated. After passing to an open cover of $X$, \cite[Theorem II.2.5.4]{ogusloggeom} provides us with a locally finite stratification $X = \bigsqcup \limits_{\sigma \in \Sigma} X_\sigma$ where 
\begin{itemize}
    \item For each $\sigma \in \Sigma$, $\overline{M}_X|_\sigma$ is constant. 
    \item The cospecialization maps for $x \in \overline{\{\xi\}} \subseteq X$ 
    \[\overline{M}_x \rightarrow \overline{M}_\xi\]
    are localizations at faces. 
\end{itemize}
The localization of a saturated monoid remains saturated \cite[Remark I.1.4.5]{ogusloggeom} and a monoid is saturated if and only if its characteristic monoid is \cite[Proposition I.1.3.5]{ogusloggeom}. We then have that $X^{fs} \subseteq X$ is locally a constructible subset which is closed under generization, and hence open \cite[Tag 0542]{sta}. 

\end{proof}

We collect several results of \cite{logcotangent} adapted to the fs setting:

\begin{lemma}[{\cite[Theorem 2.4, Proposition 2.11, Lemma 3.12]{logcotangent}}]\label{olssoncotcplxfsversion}
These statements remain true in the \textit{fs} context:
\begin{enumerate}
    \item\label{olssoncotcplxfinitediagramsoffslstrsisastack} For any finite category $\Gamma$, the fibered category $\Log^\Gamma$ of diagrams of fs log structures indexed by $\Gamma$ is an algebraic stack. 
    
    \item\label{olssoncotcplxsimplicialetalestack} The simplicial face maps $d_j : \Log^{i+1} \rightarrow \Log^i$ are strict, \'etale, and DM-type for $j \leq i$. 
    
    \item If $[1] \rightarrow \square$ avoids the initial object $0$ ( $\begin{tikzpicture}[scale = 0.3]
    \draw[dotted] (0, 0) -- (0, 1);
    \draw (0, 0) -- (1, 0);
    \draw[dotted] (1, 0) -- (1, 1);
    \draw[dotted] (0, 1) -- (1, 1);
    \end{tikzpicture} \text{ or }
    \begin{tikzpicture}[scale = 0.3]
    \draw[dotted] (0, 0) -- (0, 1);
    \draw[dotted] (0, 0) -- (1, 0);
    \draw (1, 0) -- (1, 1);
    \draw[dotted] (0, 1) -- (1, 1);
    \end{tikzpicture}$ ), it induces a strict \'etale, DM-type morphism
    \[\Log^\square \rightarrow \Log^1. \]
    
    \item\label{olssoncotcplxetalemapsfor2simplex} If $[2] \rightarrow \square$ omits either $1$ or $2$ ( $\begin{tikzpicture}[scale = 0.3]
    \draw (0, 0) -- (0, 1);
    \draw (0, 0) -- (1, 0);
    \draw (0, 1) -- (1, 0);
    \draw[dotted] (1, 0) -- (1, 1);
    \draw[dotted] (0, 1) -- (1, 1);
    \end{tikzpicture} \text{ or }
    \begin{tikzpicture}[scale = 0.3]
    \draw[dotted] (0, 0) -- (0, 1);
    \draw[dotted] (0, 0) -- (1, 0);
    \draw (0, 1) -- (1, 0);
    \draw (1, 0) -- (1, 1);
    \draw (0, 1) -- (1, 1);
    \end{tikzpicture}$ ), it induces an \'etale, DM-type morphism
    \[\Log^\square \rightarrow \Log^2.\]
    
    \item The map $\glob \subseteq \Log^\square$ is an open embedding.

    \item\label{olssonsquareconstruction} Given an fs pullback square
    \[\begin{tikzcd}
    X' \arrow[r] \arrow[d] \lpb      &X \arrow[d]      \\
    Y' \arrow[r, "q"]      &Y,
    \end{tikzcd}\]
    the associated square of stacks
    \[\begin{tikzcd}
    X' \arrow[r] \arrow[d] \arrow[phantom, very near start, "\ulcorner", dr]      &X \arrow[d]      \\
    \glob_q \arrow[r]       &\Log Y
    \end{tikzcd}\]
    is a pullback. 
\end{enumerate}
\end{lemma}

\begin{proof}

Facts \eqref{olssoncotcplxfinitediagramsoffslstrsisastack} through \eqref{olssoncotcplxetalemapsfor2simplex} are immediate by Lemma \ref{fsisopenimmersioninfinelstrs} and the analogous facts in \cite{logcotangent}. The last two follow by the same arguments applied in the fs category.

\end{proof}

\begin{remark}\label{rmk:logetfacemapslogstacks}
Apply $\Log$ once more to the map $\Log Y \to Y$: one gets 
\[d_1 : \Log^2 Y \to \Log Y \quad (M_Y \to M_0 \to M_1) \mapsto (M_Y \to M_1).\]
The result is \'etale, so the original $d_1 : \Log Y \to Y$ is log \'etale \cite[Theorem 4.6 (ii)]{logstacks}. The same reasoning concludes $d_{i+1} : \Log^{i+1} Y \to \Log^i Y$ is log \'etale in general. In summary, all the face maps are log \'etale and all but $j=i+1$ are furthermore strict \'etale. 

\end{remark}

\begin{remark}\label{squaresandglobareetaleoverthesource}
Given $q : Y' \rightarrow Y$ DM, the natural maps
    \[\glob_q \subseteq \Log^\square_q \rightarrow \Log Y'\]
    are \'etale. The second map is the product of the \'etale map
    \[\begin{tikzpicture}[scale = 0.3]
    \draw[dotted] (0, 0) -- (0, 1);
    \draw[dotted] (0, 0) -- (1, 0);
    \draw (1, 0) -- (1, 1);
    \draw (0, 1) -- (1, 1);
    \draw (0, 1) -- (1, 0);
    \end{tikzpicture}^* : \Log^\square \rightarrow \Log^2\]
    over $\Log^1$ (via \begin{tikzpicture}[scale = 0.3]
    \draw[dotted] (0, 0) -- (0, 1);
    \draw[dotted] (0, 0) -- (1, 0);
    \draw[dotted] (1, 0) -- (1, 1);
    \draw (0, 1) -- (1, 1);
    \end{tikzpicture} ) with $Y'$. 

\end{remark}

\begin{definition}\label{olssonmors}
Use Lemma \ref{olssoncotcplxfsversion}, bullet \eqref{olssonsquareconstruction} to turn one commutative square of DM maps into another: 
\[\begin{tikzcd}
X' \arrow[r] \arrow[d]      &X \arrow[d]     \\
Y' \arrow[r, "q"]       &Y.
\end{tikzcd} \quad \rightsquigarrow \quad \begin{tikzcd}
X' \arrow[r] \arrow[d]      &X \arrow[d]      \\
\Log^\square_q \arrow[r]       &\Log Y
\end{tikzcd}\]

Maps of normal sheaves 
\[\varphi : \Nl{X'/Y'} \simeq \N{X/\Log^\square_q} \rightarrow \Nl{X/Y}\]
arise from Remark \ref{squaresandglobareetaleoverthesource} and the second square. We call the composite $\varphi$ \textit{Olsson's Morphism}. 

\end{definition}

\begin{comment}
Provided the identification $N_{X/\Log Y} \simeq \Nl{X/Y}$ of \cite{onnormalsheaves}, we can describe the functor of points of $\varphi$ simply from the diagram 
\[\begin{tikzcd}
(S, \OO_{X'}|_S) \arrow[dd, hook] \arrow[r]       &X' \arrow[r] \arrow[d]      &X \arrow[d]      \\
        &\Log^\square_q \arrow[r] \arrow[d]      &\Log Y         \\
(S, \mathcal{A}) \arrow[r] \arrow[dashed, ur, "\exists !"]       &\Log Y'.
\end{tikzcd}\]
A map $S \to N_{X'/\Log Y'}$ over a map $S \to X'$ is equivalent to the data of the depicted squarezero algebra extension $(S, \mathcal{A})$ of the ringed space $(S, \OO_{X'}|_S)$. We have to extend $\Log Y'$ to \'etale-locally ringed spaces to make this rigorous. Then the unique dashed extension exists because $\Log_q^\square \to \Log Y'$ is \'etale, and the Olsson Morphism colligates the right square to this lift. 
\end{comment}

\begin{remark}\label{remarkfspullbacksinduceclimmsoflognormalsheaves}
In Definition \ref{olssonmors}, if the first square was an fs pullback square, the second factors:
\[\begin{tikzcd}
X' \arrow[rr] \arrow[d] \arrow[drr, phantom, very near start, "\ulcorner"]      & &X \arrow[d]      \\
\glob_q \arrow[r, hook]     &\Log^\square_q \arrow[r]          &\Log Y.
\end{tikzcd}\]
Since this square is a pullback, Olsson's morphism 
\[\varphi: \Nl{X'/Y'} \simeq \N{X'/\Log^\square_q} \simeq \N{X'/\glob_q} \hookrightarrow \Nl{X/Y}|_{X'}\]
is then a closed immersion. 

If $q$ or $f$ is also log flat, $\varphi$ might not be an isomorphism. See Lemmas \ref{lem:strsmloctargsmloccone}, \ref{lem:strsmlocsrcsmloccone} for the strict case.

\end{remark}

\begin{remark}\label{strictolssonmorsareordinary}
A commutative square of DM maps may be factored: 
\begin{equation}\label{eqn:commsqstrolssonmorsareordinary}
\begin{tikzcd}
X' \arrow[r] \arrow[d]      &X \arrow[d, "f"]      \\
Y' \arrow[r, "q"]      &Y.
\end{tikzcd}\quad \rightsquigarrow \quad 
\begin{tikzcd}
X' \ar[r] \ar[d]          &X \ar[d]      \\
\Log^\square_q \ar[r] \ar[d]      &\Log Y \ar[d]         \\
Y' \arrow[r]      &Y.
\end{tikzcd}
\end{equation}

This induces a commutative square of normal sheaves:
\begin{equation}\label{eqn:normalsheavesolssonvsordinary}
\begin{tikzcd}
\Nl{X'/Y'} \simeq \N{X'/\Log^\square_q} \ar[r] \ar[d] \ar[dr, phantom, "\circ"]        &\Nl{X/Y} \ar[d]           \\
\N{X'/Y'} \ar[r]       &\N{X/Y}.
\end{tikzcd}
\end{equation}
The Olsson morphisms are thereby seen to be compatible with the ordinary functoriality of the normal sheaf via the forgetful maps $\Nl{X/Y} \to \N{X/Y}$. 

Now suppose the original square \eqref{eqn:commsqstrolssonmorsareordinary} is an fs pullback:
\begin{itemize}
    \item If $q$ is strict, then $\glob_q \simeq \Log Y'$, and our fs pullback square factors as
    \[\begin{tikzcd}
    X' \arrow[r] \arrow[d] \arrow[dr, phantom, very near start, "\ulcorner"]      &X  \arrow[d]     \\
    \glob_q \simeq \Log Y' \arrow[r] \arrow[d] \arrow[dr, phantom, very near start, "\ulcorner"]        &\Log Y \arrow[d]     \\
    Y' \arrow[r]      &Y,
    \end{tikzcd}\]
    and the functor of points witnesses that \eqref{eqn:normalsheavesolssonvsordinary} is cartesian. 
    \item If instead $f$ is strict, then $X' \rightarrow \glob_q$ factors through $Y'$, and the factorization
    \[\begin{tikzcd}
    X' \arrow[r] \arrow[d] \arrow[phantom, dr, very near start, "\ulcorner"]      &X  \arrow[d]     \\
    Y' \arrow[d] \arrow[r] \arrow[phantom, dr, very near start, "\ulcorner"]      &Y \arrow[d, hook]      \\
    \glob_q \arrow[r]       &\Log Y
    \end{tikzcd}\]
    shows that the vertical arrows of \eqref{eqn:normalsheavesolssonvsordinary} are isomorphisms and the Olsson Morphism is the same as the ordinary functoriality of the Normal Sheaf. 
\end{itemize}

\end{remark}

\begin{remark}\label{rmk:compareolssonmorandlogsoneverything}
Given a commutative square 
\[\begin{tikzcd}
X' \arrow[r] \arrow[d]      &X \arrow[d]      \\
Y' \arrow[r, "q"]      &Y, 
\end{tikzcd}\]
of DM maps we can form two other commutative squares out of it: 
\[\begin{tikzcd}
X' \arrow[r] \arrow[d]      &X \arrow[d]      \\
\Log^\square_q \arrow[r]      &\Log Y,
\end{tikzcd}
\begin{tikzcd}
X' \arrow[r] \arrow[d]         &\Log X \arrow[d]         \\
\Log Y' \arrow[r]         &\Log Y.
\end{tikzcd}\]
They induce morphisms
\[\Nl{X'/Y'} \simeq \Nl{X'/\Log^\square_q} \to \Nl{X/Y}|_{X'},\]
\[\Nl{X'/Y'} \to N_{\Log X/\Log Y}|_{X'}.\]

Form the diagram 
\[\begin{tikzcd}
X' \arrow[r] \arrow[d]      &\Log X \arrow[r, "s"] \arrow[d] \pb         &X \arrow[d]      \\
\Log_q^\square \arrow[d] \arrow[r, "{\begin{tikzpicture}[scale = 0.3]
    \draw[dotted, -] (0, 0) -- (0, 1);
    \draw[dotted, -] (0, 0) -- (1, 0);
    \draw[-] (0, 1) -- (1, 0);
    \draw[-] (1, 0) -- (1, 1);
    \draw[-] (0, 1) -- (1, 1);
    \end{tikzpicture}}"] \arrow[rr, bend right, "{\begin{tikzpicture}[scale = 0.3]
    \draw[dotted, -] (0, 0) -- (0, 1);
    \draw[dotted, -] (0, 0) -- (1, 0);
    \draw[-] (1, 0) -- (1, 1);
    \draw[dotted, -] (0, 1) -- (1, 1);
    \end{tikzpicture}}"]     &\Log^2 Y \ar[r, "d_0"]       &\Log Y         \\
\Log Y'
\end{tikzcd}\]
to see that the two morphisms of normal sheaves are compatible: 
\[\Nl{X'/Y'} \simeq N_{X'/\Log_q^\square} \to N_{\Log X/\Log^2 Y}|_{X'} \subseteq \Nl{X/Y}|_{X'}.\]

\end{remark}

\begin{lemma}\label{compatibleolssonmors}
Suppose given a pair of commutative squares: 
\[\begin{tikzcd}
X' \arrow[r] \arrow[d]     &Y' \arrow[r] \arrow[d]     &Z' \arrow[d]      \\
X \arrow[r, "f"]     &Y \arrow[r, "g"]     &Z      
\end{tikzcd}\]
of DM-type maps. The diagram 
\[\begin{tikzcd}
        &\Nl{Y/Y'} \arrow[dr]      \\
\Nl{X/X'} \arrow[ur] \arrow[rr]       &       &\Nl{Z/Z'}
\end{tikzcd}\]
commutes, where all the arrows are Olsson's morphisms. 

\end{lemma}

\begin{proof}

Introduce an algebraic $X$-stack $\rectglob$, with functor of points: 
\[\left\{\begin{tikzcd}
        &\rectglob \arrow[d]        \\
T \arrow[r] \arrow[ur, dashed]       &X
\end{tikzcd}\right\} :=
\left\{\begin{tikzcd}
M_2      &M_1 \arrow[l]        &M_0 \arrow[l]       \\
M_X|_T \arrow[u]     &M_Y|_T \arrow[l] \arrow[u]        &M_Z|_T \arrow[l] \arrow[u] 
\end{tikzcd}
\begin{tikzcd}[row sep=tiny]
\text{commutative diagrams of }     \\
\text{fs log structures on }T
\end{tikzcd}\right\}\]
In other words, $\rectglob := (\Log^\square \times_{\Log^1} \Log^\square)\times_{\Log^2} X$. 

All the triangles in this diagram commute because of the definition of Olsson morphisms and the functor of points of $\N{}$: 
\[\begin{tikzcd}
        &       &\N{X/\Log^\square_{g \circ f}} \arrow[ddrr, bend left]       &       &       \\
        &       &       &       &       \\
\Nl{X'/X} \arrow[rr] \arrow[d, no head, "\sim"] \arrow[uurr, bend left, no head, "\sim"]       &       &\N{X/\rectglob} \arrow[rr] \arrow[uu] \arrow[dd] \arrow[lld]       &       &\Nl{Z'/Z}       \\
\N{X'/\Log^\square_f} \arrow[drr, bend right]        &       &       &       &       \\
        &       &\N{Y/\Log^\square_g} \arrow[uurr]       &\Nl{Y'/Y} \arrow[l, no head, "\sim"] \arrow[uur, bend right]       &       
\end{tikzcd}\]

Restricting the diagram to $\Nl{X'/X}$, $\Nl{Y'/Y}$, and $\Nl{Z'/Z}$, we get the result.

\end{proof}

\begin{proposition}\label{exactseqfornormalsheaves}
Given $X \overset{f}{\rightarrow} Y \overset{g}{\rightarrow} Z$ DM-type maps of log algebraic stacks, the Olsson Morphisms yield a complex of stacks
\[\Nl{X/Y} \rightarrow \Nl{X/Z} \rightarrow \Nl{Y/Z}|_X,\]
in that the composite factors through the vertex. 

If $h$ is smooth, $\Nl{Y/Z} = B\Tl{Y/Z}$ and rotating the triangle in the derived category yields an exact sequence of cone stacks: 
\[\Tl{Y/Z}|_X \to \Nl{X/Y} \to \Nl{X/Z}. \]

\end{proposition}

\begin{proof}

The Olsson Morphisms come about from the commutative diagram
\[\begin{tikzcd}
X \arrow[r, equals] \arrow[d]       &X \arrow[r] \arrow[d]      &Y \arrow[d]      \\
Y \arrow[r]       &Z \arrow[r, equals]      &Z
\end{tikzcd} \rightsquigarrow\begin{tikzcd}
\Nl{X/Y} \arrow[r] \arrow[d]        &\Nl{Y/Y}|_X = X \arrow[d, "0"]       \\
\Nl{X/Z} \arrow[r]        &\Nl{Y/Z}|_X.
\end{tikzcd}\]

The surjectivity, smoothness, and calculation of the fiber of $\Nl{X/Y} \rightarrow \Nl{X/Z}$ may all be checked routinely using the functor of points.

\end{proof}

\begin{remark}\label{findiagramcommcheckonabhulls}
Suppose given a (not necessarily commutative) finite diagram of cones. If the diagram induced by taking abelian hulls is commutative, so was the original. 
\end{remark}

\section{Properties of the Log Normal Cone}

We are ready to define the log normal cone. We recall the essential properties of the ordinary normal cone; the rest of the section establishes analogous properties in the log context. 

\begin{remark}
Consider a DM-type morphism $f: X \rightarrow Y$ of algebraic stacks. K. Behrend and B. Fantechi defined the Intrinsic Normal Cone \cite{intrinsic}
\[\C{f} = \C{X/Y} \subseteq \N{X/Y};\]
C. Manolache \cite{virtpb} removed their assumptions of smooth $Y$ and DM $X$. This cone has the following basic properties: 
\begin{enumerate}
    \item A commutative diagram 
    \[\begin{tikzcd}
    X' \arrow{r} \arrow[d]      &X \arrow[d, "f"]      \\
    Y' \arrow[r, "q"]      &Y
    \end{tikzcd}\]
    yields a morphism of cones $\varphi : \C{X'/Y'} \rightarrow \C{X/Y}\times_X X'$. 
    \begin{itemize}
        \item if the square was cartesian, $\varphi$ is a closed embedding. 
        \item if also $f$ or $q$ was flat, $\varphi$ is an isomorphism. 
    \end{itemize}
    
    \item\label{lciseslogcones} For a composite
    \[X \overset{f}{\rightarrow} Y \overset{g}{\rightarrow} Z\]
    \begin{itemize}
        \item if $g$ is l.c.i., $\C{X/Y} = \N{X/Y}$ and the sequence 
        \[\N{X/Y} \rightarrow \C{X/Z} \rightarrow \C{Y/Z}|_X\]
        of cone stacks is exact.
        \item if $h$ is smooth, the sequence 
        \[T_{Y/Z}|_X \rightarrow \C{X/Y} \rightarrow \C{X/Z}\]
        is exact.
    \end{itemize}
    
    \item Obstruction Theories and Gysin Pullbacks are obtained by placing the cone in a vector bundle stack $\C{X/Y} \subseteq E$ (see \cite{virtpb}, \cite{obstthies}, \cite{kreschthesis}). 
\end{enumerate}
\end{remark}

\begin{definition}[Log Intrinsic Normal Cone, Olsson Morphisms]
Let $f : X \rightarrow Y$ be a DM-type morphism of log algebraic stacks. We define the \textit{Log (Intrinsic) Normal Cone }
\[\Cl{X/Y} := \C{X/\Log Y} \subseteq \Nl{X/Y}\]

{\noindent}after \cite{loggw}. Endow it with the log structure pulled back from $X$. Given a commutative square of log algebraic stacks and its partner
\[\begin{tikzcd}
X' \arrow[r] \arrow[d]      &X \arrow[d]      \\
Y' \arrow[r, "q"]      &Y
\end{tikzcd} \quad \rightsquigarrow \quad
\begin{tikzcd}
X' \arrow[r] \arrow[d]      &X \arrow[d]      \\
\Log^\square_q \arrow[r]     &\Log Y,
\end{tikzcd}
\]

the latter induces 
\[\varphi : \Cl{X'/Y'} \simeq \C{X'/\Log^\square_q} \rightarrow \Cl{X/Y}.\]
This is again called the \textit{Olsson Morphism}.

\end{definition}

\begin{remark}\label{logconesameasusualifstrict}

The map $\Log Y \rightarrow Y$ has a section $Y \subseteq \Log Y$ which is an open immersion. This open immersion represents strict log maps to $Y$. 

As a result, if $X \rightarrow Y$ is DM and strict, $\Cl{X/Y} = \C{X/Y}$ and $\Nl{X/Y} = \N{X/Y}$. In addition, the Olsson Morphisms are the same as the ordinary functoriality of the normal cone (Remarks \ref{findiagramcommcheckonabhulls} and \ref{strictolssonmorsareordinary}). 

The Olsson Morphism of any fs pullback square is a closed immersion, because it fits into a commutative square of closed immersions from Remark \ref{remarkfspullbacksinduceclimmsoflognormalsheaves}:
\[\begin{tikzcd}
\Cl{X'/Y'} \ar[r] \ar[d, hook]       &\Cl{X/Y}|_{X'}  \ar[d, hook]       \\
\Nl{X'/Y'} \ar[r, hook]       &\Nl{X/Y}|_{X'}.         
\end{tikzcd}\]

\end{remark}

\begin{comment}
\todo{is it okay to take $E$ a vector bundle \textit{stack} instead of a vector bundle here?}

\end{comment}

\begin{remark}[Short Exact Sequences of Cone Stacks]\label{rmk:defnsesconestacks}
Recall \cite[Definition 1.12]{intrinsic}. Let $E$ be a vector bundle stack and $C, D$ cone stacks all on some base algebraic stack $X$. A composable pair of morphisms of cone stacks
\[E \to C \to D\]
is called a \textit{short exact sequence} if 
\begin{itemize}
    \item $C \to D$ is a smooth epimorphism. 
    \item The square
    \[\begin{tikzcd}
    E \times C \ar[r, "pr_2"] \ar[d, "\sigma"]      &C \ar[d]      \\
    C \ar[r]       &D,
    \end{tikzcd}\]
    where $pr_2$ is the projection and $\sigma$ the action, is cartesian.
\end{itemize}
These are equivalent to having $C \simeq E \times_X D$ locally in $X$. 

Note that this definition is \textit{fpqc}-local in the base $X$ \cite[02VL]{sta}. Another reduction we will need applies in case there is a commutative diagram of cone stacks
\[\begin{tikzcd}
E \ar[r] \ar[d]       &C \ar[r] \ar[d, "s"]      &D  \arrow[d, "t"]     \\
E' \ar[r]      &C' \ar[r]     &D'
\end{tikzcd}\]
with $E, E'$ vector bundles. If the top sequence is exact and the arrows labeled $s, t$ are smooth and surjective, then the bottom is exact. To see this, pushout along $E \to E'$ so as to assume $E = E'$ ($s, t$ remain smooth and surjective). The diagram on the left is the pullback along the smooth surjection $D' \to D$ of the one on the right:
\[\begin{tikzcd}
E \times C' \ar[r] \ar[d] \pb         &C' \ar[d]         \\
C' \ar[r]      &D'
\end{tikzcd}
\begin{tikzcd}
E \times C \ar[r] \ar[d]      &C \ar[d]      \\
C \ar[r]       &D,
\end{tikzcd}\]
and we can verify that $E \times C$ is the pullback after smooth-localizing. 
\end{remark}

\begin{proposition}\label{targetsmdescentnormalcone}
    Suppose $X \overset{f}{ \rightarrow} Y \overset{g}{\rightarrow} Z$ are DM maps between log algebraic stacks, and $g$ is log smooth. Then 
    \[\Tl{Y/Z}|_X \rightarrow \Cl{X/Y} \rightarrow \Cl{X/Z}\]
    is an exact sequence of cone stacks. 
\end{proposition}

\begin{proof}

Encode the log structures on the maps via the top row of the diagram
\[\begin{tikzcd}
X \arrow[r] \arrow[dr]       &\Log Y \arrow[d] \arrow[r] \arrow[dr, phantom, very near start, "\ulcorner"]         &\Log^2 Z \arrow[r] \arrow[d]   \arrow[dr, phantom, very near start, "\ulcorner"]     &\Log^2  \arrow[d]    \\
        &Y \arrow[r]      &\Log Z \arrow[r]     &\Log^1.
\end{tikzcd}\]

Since $Y \rightarrow \Log Z$ is smooth, $\Log Y \rightarrow \Log^2 Z$ is. Moreover, they have the same tangent bundle: 
\[\Tl{Y/Z}|_{\Log Y} = T_{Y/\Log Z}|_{\Log Y} = T_{\Log Y/\Log^2 Z}\]
since the vertical maps are log \'etale \cite[Corollary IV.3.2.4]{ogusloggeom}. 

Together with the isomorphism $\Cl{X/Z} \simeq \C{X/\Log^2 Z}$, we obtain the exact sequence. 

\end{proof}

\begin{remark}

In the proof, the composite 
\[\Cl{X/Y} \rightarrow \C{X/\Log^2 Z} \simeq \Cl{X/Z}\]
is precisely the Olsson Morphism. This is immediate from the diagram: 
\[\begin{tikzcd}
X \arrow[r, equals] \arrow[d]       &X \arrow[d]      \\
\Log^\square_q  \arrow[d]\arrow[r]      &\Log^2 Z \arrow[d]       \\
\Log Y \arrow{ur}      &\Log Z.
\end{tikzcd}\]

\end{remark}

\begin{remark}\label{axiomaticcharacterizationoflognormalcones}

The introduction promised three characterizations of $\Cl{X/Y}$. 

The log intrinsic normal cone is characterized by the strict case of Remark \ref{logconesameasusualifstrict} and the log \'etale case of Proposition \ref{targetsmdescentnormalcone}. This is because any map $X \to Y$ factors into the strict map $X \to \Log Y$ composed with the log \'etale map $\Log Y \to Y$ (Remark \ref{rmk:logetfacemapslogstacks}).

We can unpack this definition locally using charts. Suppose a morphism has a global fs chart by Artin Cones: 
\[\begin{tikzcd}
X \arrow{r} \arrow{d}       &Y \arrow{d}      \\
A_P \arrow{r}     &A_Q. 
\end{tikzcd}\]
The morphism $A_P \rightarrow A_Q$ is log \'etale \cite[Corollary 5.23]{logstacks}. Let $W = A_P \times^\ell_{A_Q} Y$ denote the fs pullback, so that $X \rightarrow Y$ factors through a strict map to $W$ and $W$ is log \'etale over $Y$. We immediately get
\[\Cl{X/Y} = \C{X/W}.\]
The reader may be reassured by working locally with this definition. If the reader wants instead to work with charts $\Spec (P \rightarrow \CC[P])$ in the traditional sense, then log \'etaleness is no longer immediate and we must check Kato's Criteria \cite[Corollary IV.3.1.10]{ogusloggeom}.

Recall Construction \ref{katofactorization} -- after localizing in the \'etale topology, we obtain a factorization of any map $X \rightarrow Y$ as a \textit{strict} closed immersion followed by a log smooth map 
\[X \subseteq X_\theta \rightarrow Y.\]
Proposition \ref{targetsmdescentnormalcone} therefore locally provides a presentation of the log normal cone: 
\[\Cl{X/Y} = [\C{X/X_\theta}/\Tl{X_\theta/Y}].\]
\end{remark}

\begin{lemma}\label{ftfactor}
Given a DM map $f : X \rightarrow Y$ of log algebraic stacks with $X$ quasicompact, the map 
\[X \rightarrow \Log Y\]
factors through an open quasicompact subset $U \subseteq \Log Y$. 
\end{lemma}

Our applications require openness; otherwise the lemma is trivial. 

\begin{proof}
The claim is \'etale-local in $Y$ and $X$ because $X$ is quasicompact. We can thereby assume we have a global chart
\[\begin{tikzcd}
X \ar[r] \ar[d]       &Y \ar[d]      \\
A_P \ar[r]         &A_Q.
\end{tikzcd}\]

The map $A_P \times_{A_Q} Y \to \Log Y$ is \'etale \cite[Corollary 5.25]{logstacks} and $X$ factors through its open, quasicompact image. 

\end{proof}

\begin{remark}
This lemma ensures that any DM map $X \to Y$ of log stacks with $X$ quasicompact factors through $X \to U \to Y$ with $X \to U$ strict, $U$ quasicompact, and $U \to Y$ log \'etale. 
\end{remark}

\begin{example}\label{exaffinediagonal}
We provide an example of Construction \ref{katofactorization} and Remark \ref{katofactorizationwithsurjectivemonoids}. 

Consider the diagonal morphism $\mathbb{A}^1 \overset{\Delta}{\rightarrow} \mathbb{A}^2$. The addition map $\NN^2 \overset{+}{\rightarrow} \NN$ gives a chart for $\Delta$. 

Denote by $B$ the log blowup of $\mathbb{A}^2$ at the ideal $I \subseteq M_{\mathbb{A}^2}$ generated by $\NN^2 \setminus \{0\} \subseteq \NN^2$. The pullback $\Delta^*I$ is generated by the image of the composite
\[\NN^2 \setminus \{0\} \subseteq \NN^2 \overset{+}{\rightarrow} \NN.\]
The pullback is generated globally by a single element and so $\Delta$ factors through the log blowup $B$. 

Name the generators $\NN^2 = \NN e \oplus \NN f$. The log blowup $B$ is covered by two affine opens $D_+(e)$ and $D_+(f)$, on which $e$ and $f$ are invertible.

On the chart $D_+(e)$, the morphism $\mathbb{A}^1 \rightarrow B$ looks like
\[\begin{tikzcd}
\NN \arrow{d}     &\NN e \oplus \NN (f-e) \arrow{d} \arrow{l}     \\
\CC[t]        &\CC[x, \dfrac{y}{x}] \arrow{l}.
\end{tikzcd}\]
The horizontal morphisms send $f-e \mapsto 0$ and $\dfrac{y}{x} \mapsto 1$. Because $(f-e)$ maps to $1 \in \CC[t]$, the composite 
\[\NN e \oplus \NN (f-e) \rightarrow \NN \rightarrow \CC[t]\]
is another chart for the same log structure on $\mathbb{A}^1$. This means that $\mathbb{A}^1 \rightarrow D_+(e)$ is strict. The same discussion applies to $D_+(f)$. In the tropical picture \cite[\S 2]{tropicalcurvesmodulicones}, we subdivided $\mathbb{A}^2$ at the image of the ray corresponding to $\mathbb{A}^1$:
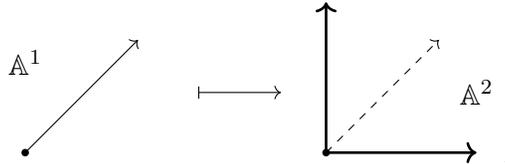
\[\begin{tikzpicture}
\draw[->] (-3, 0) -- (-1.5, 1.5);
\node at (-3, 1.2){$\mathbb{A}^1$};
\fill (-3, 0) circle (.05);
\draw[->, line width=1pt] (1, 0) -- (1, 2);
\draw[->, line width=1 pt] (1, 0) -- (3, 0);
\node at (3, .8){$\mathbb{A}^2$};
\fill (1, 0) circle (.05);
\draw[|->] (-.7, .8) -- (.4, .8);
\draw[->, dashed] (1, 0) -- (2.5, 1.5);
\end{tikzpicture}.\]

\end{example}

\begin{proposition}\label{lciseschangesource}
Consider DM-type morphisms $X \overset{f}{\rightarrow} Y \overset{g}{\rightarrow} Z$ between log algebraic stacks. If $\Cl{X/Y} = \Nl{X/Y}$, then
\[\Nl{X/Y} \rightarrow \Cl{X/Z} \rightarrow \C{\Log Y/\Log Z}|_X\]
is an exact sequence of cone stacks. 
\end{proposition}

\begin{proof}

Compare \cite[Proposition 3.14]{intrinsic}.

By Proposition \ref{exactseqfornormalsheaves} and Remark \ref{rmk:compareolssonmorandlogsoneverything}, this sequence composes to zero. Remark \ref{rmk:defnsesconestacks} allows us to repeatedly \textit{fpqc}-localize in $X$ to check exactness of such a sequence. Localizing along strict smooth covers of $Z$ and strict \'etale covers of $X$ and $Y$ ensures that the normal cones and sheaf pull back. Reduce to the case where $X$, $Y$, and $Z$ are affine log schemes and the map $Y \to Z$ admits a global fs chart. We are therefore in the situation of Construction \ref{katofactorization}.

\textbf{Reduction to $g : Y \to Z$ Strict}

Factor $Y \to Z$ into a strict closed immersion composed with a log smooth map:
\[Y \subseteq W \twoheadrightarrow Z.\]

We obtain a diagram
\[\begin{tikzcd}
        &\Tl{W/Z}|_X \ar[d] \ar[r, equals]       &T_{\Log W/\Log Z}|_X     \ar[d]  \\
\N{X/Y} \ar[r] \ar[d, equals]        &\Clstrict{X/W} \ar[r] \ar[d] \pb       &\C{\Log Y/\Log W}|_X \ar[d]        \\
\N{X/Y} \ar[r]        &\Cl{X/Z} \ar[r]       &\C{\Log Y/\Log Z}|_X.        \\
\end{tikzcd}\]
Observe that the diagram commutes -- the morphism $\Tl{W/Z}|_X \to \Clstrict{X/W}$ in the proof of Proposition \ref{targetsmdescentnormalcone} factors through an identification $\Tl{W/Z}|_{\Log W} \simeq \Tl{\Log W/\Log^2 Z}$. Because $\Log W \to W$ is log \'etale, the two tangent spaces are isomorphic \cite[IV.3.2.4]{ogusloggeom}. Thus the right square is a pullback. The vertical maps of cones are smooth surjections, so it suffices to show the middle row is exact as in Remark \ref{rmk:defnsesconestacks}. We may thereby assume $W = Z$ and $g : Y \to Z$ is a strict closed immersion.

\textbf{Reduction to $f: X \to Y$ Strict}

Use Construction \ref{katofactorization} again to factor $X \to Z$ as a strict closed immersion composed with a log smooth map $X \subseteq W \twoheadrightarrow Z$. The map $X \to W' := W \times_Z Y$ is again a strict closed immersion: 
\begin{equation}\label{eqn:lcisesstrfactndiagram}
\begin{tikzcd}
X \ar[r, hook] \ar[dr, bend right]       &W' \ar[r, hook] \ar[d] \lpb         &W \ar[d]     \\
        &Y \ar[r, hook]      &Z.
\end{tikzcd}
\end{equation}

Because the top row is strict, $X \to \Log W'$ factors through the open subset $W' \subseteq \Log W'$ and
\[\C{\Log W'/\Log W}|_X = \C{\Log W'/\Log W}|_{W'}|_X = \Cl{W'/W}|_X = \C{W'/W}|_X.\]

The fs pullback square in \eqref{eqn:lcisesstrfactndiagram} also induces a cartesian square of stacks: 
\[\begin{tikzcd}
\Log W' \ar[r] \ar[d] \pb         &\Log W \ar[d]         \\
\Log Y \ar[r]      &\Log Z
\end{tikzcd}\]
with $\Log W \to \Log Z$ smooth. This reveals that 
\[\C{\Log Y/\Log Z}|_{\Log W'} = \C{\Log W'/\Log W}.\] 
Putting this together with the above, we have computed
\[\C{\Log Y/\Log Z}|_X = \C{W'/W}|_X.\]

The factorization \eqref{eqn:lcisesstrfactndiagram} gives a diagram
\[\begin{tikzcd}
\Tl{W'/Y}|_X \ar[r, equals] \ar[d]        &\Tl{W/Z}|_X \ar[d]        \\
N_{X/W'} \ar[r] \ar[d]        &\Clstrict{X/W} \ar[r] \ar[d]       &\C{W'/W}|_X \ar[d, equals]       \\
N_{X/Y} \ar[r]         &\Cl{X/Z} \ar[r]       &\C{\Log Y/\Log Z}|_X
\end{tikzcd}\]

The composable vertical arrows are the quotients of Proposition \ref{targetsmdescentnormalcone}, so the bottom row will be exact if we show the middle row is. The middle row is exact by a relative form of the original \cite[Proposition 3.14]{intrinsic}. 

\end{proof}

\begin{remark}\label{rmk:natltyofsessandotherdisttriangle}

The exact sequences of cone stacks in Propositions \ref{targetsmdescentnormalcone}, \ref{lciseschangesource} are natural in morphisms of composable pairs of arrows. 

There is a version of Proposition \ref{lciseschangesource} for log cotangent complexes that we will use once later on. From any composable pair $X \to Y \to Z$, we get $X \to \Log Y \to \Log Z$ and $X \to \Log Y \to \Log^2 Z$. Both result in the same distinguished triangle: 
\[\ccx{\Log Y/\Log Z}|_{X} \to \lccx{X/Z} \to \lccx{X/Y} \to.\]
of \cite[8.10]{sheavesonartinstacks}. 

\end{remark}

In the next example, the log normal cone differs from the ordinary scheme-theoretic one.

\begin{example}\label{extwoblowupsofpoint}
In Example \ref{exaffinediagonal}, we considered the log blowup $B$ of $\mathbb{A}^2$ at the origin and the diagonal map. Pull back to get the identity log blowup of $\mathbb{A}^1$:
\[\begin{tikzcd}
\mathbb{A}^1 \arrow{r} \arrow[d, equals] \lpb        &B \arrow{d}       \\
\mathbb{A}^1 \arrow{r}        &\mathbb{A}^2.
\end{tikzcd}\]

Let $\point_\NN, \point_{\NN^2}$ both be $\Spec \CC$, with log structures coming from $\NN$ and $\NN^2$, respectively. Then the inclusions of the origins $\point_\NN \in \mathbb{A}^1$ and $\point_{\NN^2} \in \mathbb{A}^2$ are strict. 

Take the pullback of the above diagram along the inclusion $\point_{\NN^2} \in \mathbb{A}^2$: 
\[\begin{tikzcd}
\point_\NN \arrow[d, equals] \arrow[r] \lpb       &D \arrow[d]       \\
\point_\NN \arrow[r]       &\point_{\NN^2}.
\end{tikzcd}\]

{\noindent}The map $D \rightarrow \point_{\NN^2}$ is the exceptional divisor of $B$, which is $\mathbb{P}^1$ with log structure $\overline{M}_x = \NN^2$ at the intersections with the axes and $\overline{M}_x = \NN$ elsewhere. 

To see the log normal cone differ from the ordinary one, compute the normal cones of the arrows in this square: $\Cl{\point_{\NN}/\point_\NN} = \point$, $\Cl{\point_\NN/\point_{\NN^2}} = \Cl{\point_\NN/D} = \mathbb{A}^1$, and $\Cl{D/\point_{\NN^2}} = \mathbb{P}^1$. Although $\point_\NN$ and $\point_{\NN^2}$ have the same underlying scheme, the log normal cones of $\point_{\NN}$ over them are different. 

\end{example}

\begin{remark}\label{rmk:stretalecaseoflciseschangesource}
A handy consequence of Proposition \ref{lciseschangesource} is that, if $Y \to Z$ is a DM-type morphism between log algebraic stacks and $Y' \to Y$ is a \textit{strict} \'etale map, then 
\[\Cl{Y'/Z} \simeq \Cl{Y/Z}|_{Y'}.\]
This is \textit{not} true without the strictness assumption. This is the observation of W. Bauer precluding the existence of a log cotangent complex with all its desiderata (see \cite[\S 7]{logcotangent}). 

In general, it need only be a closed immersion. This is because 
\[\Cl{Y'/Z} \simeq \C{\Log Y/\Log Z}|_{Y'} \subseteq N_{\Log Y/\Log Z}|_{Y'} \subseteq \Nl{Y/Z}|_{Y'}\] 
is a closed immersion which factors through $\Cl{Y/Z}|_{Y'}$, as in Remark \ref{rmk:compareolssonmorandlogsoneverything}.

For a single example, take the log blowup $B \to \mathbb{A}^2$ of the origin $\point \in \mathbb{A}^2$. The pullback defines a strict pullback square: 
\[\begin{tikzcd}
D \ar[r] \ar[d] \lpbstrict       &B \ar[d]       \\
\point \ar[r]       &\mathbb{A}^2.
\end{tikzcd}\]

Because the horizontal morphisms are strict, their log normal cones coincide with the ordinary ones. Log blowups are log \'etale, so we would erroneously be led to conclude that 
\[\C{D/B} \overset{?}{=} \C{\point/\mathbb{A}^2}|_D.\]
The inclusion $D \subseteq B$ is regular, and so is $\point \in \mathbb{A}^2$, so the normal cones and normal sheaves agree: 
\[\N{D/B} = \OO_B(D)|_B\]
\[\N{\point/\mathbb{A}^2}|_D = \mathbb{A}^2_D.\]
The dimensions are different, so they can't be equal.

\end{remark}

\begin{comment}
\begin{proposition}\label{flatconeisom}
Consider an fs pullback square of DM-type morphisms of log algebraic stacks:
\begin{equation}\label{sq:flatconeisom}
\begin{tikzcd}
X' \ar[r, "p"] \ar[d, "f'", swap] \lpb      &X \ar[d, "f"]      \\
Y' \ar[r, "q", swap]      &Y
\end{tikzcd}
\end{equation}
If $q$ or $f$ is log flat, then the closed immersion 
    \[i : \Cl{X'/Y'} \hookrightarrow \C{\Log X/\Log Y}|_{X'}\]
is an isomorphism. 
\end{proposition}

\todo{new}
\begin{proof}[Proof 1: is this okay?]

Apply $\Log$ to get a cartesian square:
\[\begin{tikzcd}
\Log X' \ar[r] \ar[d] \pb         &\Log X \ar[d]         \\
\Log Y' \ar[r]         &\Log Y.
\end{tikzcd}\]
The map $\Log X \to \Log Y$ is now genuinely flat, so 
\[\C{\Log X/\Log Y}|_{\Log X'} = \C{\Log X'/\Log Y'}.\]
We pull back to $X'$ to obtain the theorem. 

\end{proof}
\end{comment}

{\color{\newstuffcolor}
\begin{lemma}\label{lem:strsmloctargsmloccone}
Suppose given a strict pullback square
\[\begin{tikzcd}
X' \ar[r] \ar[d] \lpbstrict      &X \ar[d]      \\
Y' \ar[r, "q"]      &Y
\end{tikzcd}\]
of DM-type morphisms between log algebraic stacks for which $q$ is strict and smooth. Then the Olsson Morphism
\[\Cl{X'/Y'} \overset{\sim}{\to} \Cl{X/Y}|_{X'}\]
is an isomorphism.
\end{lemma}

\begin{proof}

We first note that the Olsson Morphism $\Nl{X'/Y'} \to \Nl{X/Y}|_{X'}$ on log normal sheaves is an isomorphism. This is clear from the $q$ strict pullback part of Remark \ref{strictolssonmorsareordinary} and the fact that the ordinary normal sheaves are isomorphic.

Now we know that the morphism of cones $\Cl{X'/Y'} \to \Cl{X/Y}|_{X'}$ is a closed immersion, and it suffices to show that it is moreover smooth and surjective. We express this map as a composite
\[\Cl{X'/Y'} \to \Cl{X'/Y} \to \Cl{X/Y}|_{X'}.\]
Proposition \ref{targetsmdescentnormalcone} asserts that the first map is smooth and surjective and Proposition \ref{lciseschangesource} says the same for the second. 

\end{proof}

\begin{lemma}\label{lem:strsmlocsrcsmloccone}
Suppose given a pair of fs pullback squares
\[\begin{tikzcd}
\widtilde{X'} \ar[r] \ar[d] \lpbstrict      &\widtilde{X} \ar[d, "z"]        \\
X' \ar[r] \ar[d] \lpb      &X \ar[d]      \\
Y' \ar[r]      &Y
\end{tikzcd}\]
of DM-type morphisms between log algebraic stacks for which $z$ is strict and smooth. Then the diagram of log normal cones 
\[\begin{tikzcd}
\Cl{\widtilde{X'}/Y'} \ar[r] \ar[d, "s'"] \pb           &\Cl{\widtilde{X}/Y} \ar[d, "s"]        \\
\Cl{X'/Y'} \ar[r]          &\Cl{X/Y}
\end{tikzcd}\]
is cartesian and the arrows $s, s'$ are smooth epimorphisms. 

\end{lemma}

\begin{proof}

Proposition \ref{lciseschangesource} provides a map of short exact sequences of cone stacks: 
\[\begin{tikzcd}
B\Tl{\widtilde{X'}/X'} \ar[r] \ar[d, equals]       &\Cl{\widtilde{X'}/Y'} \ar[r, "t'"] \ar[d, hook] \ar[dr, very near start, phantom, "\ulcorner"]     &\Cl{X'/Y'}|_{\widtilde{X'}} \ar[d, hook]        \\
B\Tl{\widtilde{X}/X}|_{\widtilde{X'}} \ar[r]        &\Cl{\widtilde{X}/Y}|_{\widtilde{X'}} \ar[r, "\widtilde{t}"]     &\Cl{X/Y}|_{\widtilde{X'}}.        
\end{tikzcd}\]
Witness that the right square is cartesian because \cite{logcotangent} 
\[\Tl{\widtilde{X'}/X'} = \Tl{\widtilde{X}/X}|_{\widtilde{X'}}\]
and that the arrows $t', \widtilde{t}$ are clearly smooth epimorphisms. The arrow $\widtilde{t}$ is pulled back from the smooth epimorphism $t : \Cl{\widtilde{X}/Y} \to \Cl{X/Y}|_{\widtilde{X}}$, so we have the top pullback square
\[\begin{tikzcd}
\Cl{\widtilde{X'}/Y'} \ar[r] \ar[d, "t'"] \pb   \ar[dd, bend right=60, "s'", swap]          &\Cl{\widtilde{X}/Y} \ar[d, "t"] \ar[dd, bend left=60] \ar[phantom, "s", dr]    \\
\Cl{X'/Y'}|_{\widtilde{X'}} \ar[r] \ar[d] \pb          &\Cl{X/Y}|_{\widtilde{X}} \ar[r] \ar[d] \pb   &\widtilde{X}\ar[d]       \\
\Cl{X'/Y'} \ar[r]      &\Cl{X/Y} \ar[r]       &X
\end{tikzcd}\]
The composite vertical rectangle of cones is the diagram we are after, and so the fact that this square is cartesian is clear. It remains only to note the bent arrows $s, s'$ are smooth epimorphisms because they are the composites of $t, t'$ with pullbacks of the smooth epimorphism $\widtilde{X} \to X$.

\end{proof}}

\section{Log Intersection Theory}

The Log Intersection Theory package is defined the same way as usual \cite{virtpb}, mutatis mutandis.

\begin{definition}[Log Perfect Obstruction Theory]\label{lpotdefinition}
Define a \textit{Log Perfect Obstruction Theory} (hereafter ``\lpot'') for a DM-type morphism $f : X \rightarrow Y$ to be a closed immersion of cone stacks
\[\Cl{X/Y} \subseteq E \quad \quad \quad (\text{equiv. }\Nl{X/Y} \subseteq E)\]
of the log normal cone into a vector bundle stack $E$. 

Given an fs pullback square
\[\begin{tikzcd}
X' \arrow[r] \arrow[d, swap, "f'"] \lpb      &X \arrow[d, "f"]      \\
Y' \arrow[r]      &Y
\end{tikzcd}\]
and a \lpot \:$\Cl{X/Y} \subseteq E$ for $f$, the Olsson Morphism 
\[\Cl{X'/Y'} \overset{\varphi}{\hookrightarrow} \Cl{X/Y}|_{X'} \subseteq E|_{X'}\]
defines a ``Pullback'' \lpot.

A related notion of ``Pullback'' \lpot arises when $X' \to X$ is log \'etale and $f : X \to Y$ any DM-type map. Then Remark \ref{rmk:stretalecaseoflciseschangesource} shows the map 
\[\Cl{X'/Y} \to \Cl{X/Y}|_{X'}\]
is a closed immersion, and we can compose with an obstruction theory for $f$ to get one for the composite $X' \to X \to Y$.

Given a \lpot\, $\Cl{X/Y} \subseteq E$ for some $f$, suppose $X$ has a stratification by global quotient stacks and $Y$ is log smooth and equidimensional. Then \cite[Proposition 5.3.2]{kreschthesis} gives us a unique cycle 
\[\lvfc{X}{E} \in A_*X\]
which pulls back to the class $[\Cl{X/Y}] \in A_*E$. This class is called the \textit{Log Virtual Fundamental Class} (hereafter ``\lvirt'').
\end{definition}

\begin{remark}
When $\Log Y$ is equidimensional, so is $\Cl{X/Y}$. The correct definition of the \lvirt\, requires that the cone be equidimensional. If $Y$ is log smooth, $Y \subseteq \Log Y$ is dense. If $Y$ is also equidimensional, we get that $\Log Y$ is. This explains our assumptions in Definition \ref{lpotdefinition}. We don't include these assumptions in the definition of a \lpot\, only because we may have Log Gysin maps more generally. 
\end{remark}

\begin{definition}[Log Gysin Map]
Suppose a DM-type $f: X \rightarrow Y$ has a \lpot\, $\Cl{X/Y} \subseteq E$. 
Given a DM-type log map $k : V \rightarrow Y$ with $V$ log smooth and equidimensional, form the fs pullback: 
\[\begin{tikzcd}
W \arrow[d] \arrow[r] \lpb       &V \arrow[d, "k"]      \\
X \arrow[r, "f"]       &Y
\end{tikzcd}\]
The embedding 
\[\Cl{W/V} \subseteq \Cl{X/Y}|_W \subseteq E|_W\]
results in a class
\[[\Cl{W/V}, E] \in A_*W.\]
Mimicking \cite{virtpb}, we call this ``map'' 
\[f^! = f_E^!\] 
the \textit{Log Gysin Map}. 
\end{definition}

\begin{remark}\label{rmk:cptblitydatumextendobstthies}
Consider a DM-type morphism $f : X \to Y$ of log algebraic stacks. The cartesian square
\[\begin{tikzcd}
\Log X \ar[r, "s"]  \ar[d] \pb     &X \ar[d]      \\
\Log^2 Y \ar[r, "d_0"]        &\Log Y
\end{tikzcd}\]
from Remark \ref{rmk:compareolssonmorandlogsoneverything} results in a closed embedding 
\[\C{\Log X/\Log Y} \simeq \C{\Log X/\Log^2 Y} \subseteq \Cl{X/Y}|_{\Log X}\]
which we use to canonically extend an obstruction theory $\Cl{X/Y} \subseteq E$ to a closed embedding
\[\C{\Log X/\Log Y} \subseteq E|_{\Log X}.\]

Now suppose given a composable pair $X \overset{f}{\to} Y \overset{g}{\to} Z$ as above and equip $f, g$ with \lpot's: 
\[\Cl{X/Y} \subseteq F, \quad \quad \quad \Cl{Y/Z} \subseteq G.\]

Define a \textit{compatibility datum} for such a pair to be a traditional compatibility datum \cite[Definition 4.5]{virtpb} for 
\[X \overset{f}{\rightarrow} \Log Y \overset{g}{\rightarrow } \Log^2 Z,\]
endowing $\Log Y \to \Log^2 Z$ with the extended obstruction theory
\[\C{\Log Y /\Log^2 Z} \simeq \Cl{Y/Z}|_{\Log Y} \subseteq G|_{\Log Y}.\]
\end{remark}

We offer a couple of basic remarks about our definitions before the examples and theorems.

\begin{remark}

The map $f^!$ just defined takes in log smooth equidimensional stacks DM over $Y$ and produces classes in certain Chow Groups. We do not know whether this operation may be extended to the ``Log Chow'' groups of \cite{logchowrecentpaper}. 

\end{remark}

\begin{remark}\label{gysinmapofpullbackisthesame}
Given an fs pullback square
\[\begin{tikzcd}
X' \arrow[r] \arrow[d, swap, "f'"] \lpb      &X \arrow[d, "f"]      \\
Y' \arrow[r]      &Y
\end{tikzcd}\]
of DM maps where $f$ has a \lpot\: $\Cl{X/Y} \subseteq E$, endow $f'$ with the Pullback \lpot. Then 
\[f^! = f'^!\]
when applied to log smooth, equidimensional log schemes over $Y'$. 
\end{remark}

\begin{remark}\label{loglcigysinpreservesfundclass}
If $\Cl{X/Y} = \Nl{X/Y}$ for a DM morphism $f : X \rightarrow Y$, we can take $E = \Nl{X/Y}$ as our obstruction theory.
If $X, Y$ are equidimensional and $Y$ is log smooth, unwinding definitions shows
\[f^!(Y) = [X],\]
where $[X]$ is the fundamental class of $X$. 
\end{remark}

\begin{remark}\label{rmk:loggysindoesntcommutewpfwd}

Log Gysin Maps don't commute with pushforward: Let 
\[\begin{tikzcd}
X' \arrow[r, "p"] \arrow[d, "f'", swap] \lpb      &X \arrow[d, "f"]      \\
Y' \arrow[r, "q"]      &Y
\end{tikzcd}\]
be an fs pullback square. Endow $f : X \to Y$ with a \lpot\, $\Cl{X/Y} \subseteq E$ and give $f'$ the pullback obstruction theory. Then the usual equality \cite[Theorem 4.1 (i)]{virtpb} can fail:
\[f^! q_* \neq p_* f'^!.\]
Take the square of Example \ref{extwoblowupsofpoint}
\[\begin{tikzcd}
\point_\NN \arrow[d, equals] \arrow[r] \lpb       &D \arrow[d]       \\
\point_\NN \arrow[r]       &\point_{\NN^2}
\end{tikzcd}\]
and apply both operations to $[\point_\NN]$ for a counterexample. 

\end{remark}

\begin{remark}\label{rmk:vfcsdontpfwd}

Virtual Fundamental Classes don't push forward along log blowups: Let $X \to F$ be the morphism from a stack $X$ to its Artin Fan (the reader may take a traditional chart instead of $F$). Choose a finite subdivision $\widhat{F} \to F$, and form the fs pullback:
\[\begin{tikzcd}
\widhat{X} \arrow[r] \arrow[d, "p", swap] \lpbstrict          &\widhat{F} \arrow[d]         \\
X \arrow[r]       &F.
\end{tikzcd}\]
Suppose given a map $f : X \to Y$ with a \lpot\, $\Cl{X/Y} \subseteq E$ and equip $f \circ p : \widhat{X} \to Y$ with the pullback obstruction theory
\[\Cl{\widhat{X}/Y} \subseteq \Cl{X/Y}|_{\widhat{X}} \subseteq E|_{\widhat{X}}.\]

Then possibly
\[p_* \lvfc{\widhat{X}}{E} \neq \lvfc{X}{E}.\]
A counterexample is again given by $p : D \to \point_{\NN^2}$, $f : \point_{\NN^2} \longequals \point_{\NN^2}$ as in Example \ref{extwoblowupsofpoint}: $p_*[\mathbb{P}^1] = 0$ for dimension reasons. 

\end{remark}

The rest of this section and the next should reassure the disheartened reader that commonsense fomulas of ordinary intersection theory do remain true in the log setting. We regard Remarks \ref{rmk:loggysindoesntcommutewpfwd}, \ref{rmk:vfcsdontpfwd} as defects of the usual notion of pushforward $p_*$ in the log setting. {\color{\newstuffcolor}The morphisms $\point_\NN \to D$, $D \to \point_{\NN^2}$ of Example \ref{extwoblowupsofpoint} are monomorphisms in the fs category, and $\point_\NN \to \point_{\NN^2}$ should be a cycle of \textit{dimension one} in the ``two dimensional'' log point $\point_{\NN^2}$. }

The paper \cite{logchowrecentpaper} introduces log chow groups to correct this defect, in particular via suitable notions of dimension and degree. See also \cite{propermonomsoflogschemesmochizuki}. We are eager to see which of our results may be extended using this improved technology.

For now, we content ourselves to use the observation of \cite[Proposition 4.3]{niziol} that log blowups are birational if the target is log smooth. We will use it to prove that weaker forms of the na\"ive guesses of Remarks \ref{rmk:loggysindoesntcommutewpfwd}, \ref{rmk:vfcsdontpfwd} do hold true, as well as straightforward commutativity of the Gysin Maps.

We will need to use Costello's notion of ``pure degree $d$'' \cite[before Theorem 5.0.1]{costello} to make sense of pushforward on the level of cycles, given by cones embedded in vector bundles. The next theorem allows us to check statements about \lvirt's after a log blowup if the target is log smooth. Its statement and proof are similar to \cite{birationalinvarianceabramovichwise}.

\begin{theorem}\label{pfwdoflvfcslblowup}
Suppose given a DM-type map $f: X \rightarrow Y$ between locally noetherian algebraic stacks locally of finite type over $\CC$ where $Y$ is log smooth and equidimensional. Endow $f$ with a \lpot\,$E$ and let $X \rightarrow F$ be any DM morphism to an Artin Fan. Take the fs pullback along a finite subdivision 
\begin{equation}\label{eqn:pfwdoflvfcslblowuplblowupsquare}
\begin{tikzcd}
\widhat{X} \arrow[r] \arrow[d, swap, "p"] \lpb     &\widhat{F} \arrow[d]        \\
X \arrow[r]       &F.
\end{tikzcd}
\end{equation}

Endow $f \circ p$ with the pullback \lpot 
\[\Cl{\widhat{X}/Y} \subseteq \Cl{X/Y}|_{\widhat{X}} \subseteq E|_{\widhat{X}}.\]

Then 
\[p_* \lvfc{\widhat{X}}{E} = \lvfc{X}{E}\]
\end{theorem}

\begin{proof}

We will actually show that the map 
\[t : \Cl{\widhat{X}/Y} \to \Cl{X/Y}\]
is of pure degree one. Then the pushforward $A_*E|_{\widhat{X}} \to A_*E$ sends the class of one cone to the other, and ``intersecting with the zero section'' gives the equality of VFC's. 

We will reduce to the case where $X \to F$ is strict. The statement ``$t$ is of pure degree one'' may be verified \'etale-locally in $X$, as we now argue. 

Given a strict \'etale cover $X' \to X$, write $\widhat{X'} := \widhat{X} \times_X X'$. We have a pullback diagram
\[\begin{tikzcd}
\Cl{\widhat{X'}/F} \ar[r, "t'"] \ar[d] \pb      &\Cl{X'/F} \ar[r] \ar[d] \pb      &X' \ar[d]         \\
\Cl{\widhat{X}/F} \ar[r, "t"]       &\Cl{X/F} \ar[r]       &X,
\end{tikzcd}\]
as in Remark \ref{rmk:stretalecaseoflciseschangesource}. Since $X' \to X$ is \'etale, the other vertical arrows are as well. The property ``pure degree one'' is smooth-local in the target, so $t$ has it if $t'$ does. 

Now \'etale-localize in $X$ so that $X \to F$ factors through a chart $X \to F_X \to F$ for $X$. Take the fs pullback along the subdivision $\widhat{F} \to F$: 
\[\begin{tikzcd}
\widhat{X} \ar[r] \ar[d] \lpbstrict         &\widhat{F}_X \ar[r] \ar[d] \lpb       &\widhat{F} \ar[d]         \\
X \ar[r]      &F_X \ar[r]        &F.
\end{tikzcd}\]
We can then replace $F$ by $F_X$ in the proof of the theorem and assume $X \to F$ is strict.

Apply the proof of Costello's Formula \cite[Theorem 5.0.1]{costello} to \eqref{eqn:pfwdoflvfcslblowuplblowupsquare} to conclude 
\[t : \Cl{\widhat{X}/\widhat{F}} \rightarrow \Cl{X/F}\]
is of pure degree one, since $\widhat{F} \to F$ is birational. 

Expanding upon \eqref{eqn:pfwdoflvfcslblowuplblowupsquare}: 
\[\begin{tikzcd}
\widhat{X} \arrow[r] \arrow[d] \lpb     &\widhat{F} \times Y \arrow[r] \lpb \arrow[d]      &\widhat{F} \arrow[d]    \\
X \arrow[r]      &F \times Y \arrow[r] \arrow[d]     &F      \\
        &Y,
\end{tikzcd}\]

we get a map of exact sequences of cone stacks:
\[\begin{tikzcd}
\Tl{Y}|_{\widhat{X}} \arrow[r] \arrow[d]     &\Cl{\widhat{X}/\widhat{F} \times Y} \pb  \arrow[r] \arrow[d, "\widhat{t}"]    &\Cl{\widhat{X}/\widhat{F}} \arrow[d, "t"]       \\
\Tl{Y}|_X \arrow[r]     &\Cl{X/F \times Y}  \arrow[r]    &\Cl{X/F}.       
\end{tikzcd}\]

After pulling the bottom row back to $\widhat{X}$, we get the identity on tangent bundles and see that the right square is a pullback. Since the property``of pure degree one'' pulls back along smooth maps, the quotient maps in exact sequences of cone stacks are smooth, and $t$ is pure degree one, $\widhat{t}$ is also pure degree one. Because $F, \widhat{F}$ are log \'etale over a point, $\Cl{\widhat{X}/\widhat{F} \times Y} = \Cl{\widhat{X}/\widhat{Y}}$ and $\Cl{X/F \times Y} = \Cl{X/Y}$, so the claim is proven. 

\end{proof}

\begin{example}\label{extwoblowupsreprisecautionpfwdlvfcs}
One must be cautious, for Theorem \ref{pfwdoflvfcslblowup} is false without the assumption that $Y$ is log smooth. Recall the exceptional divisor $D \rightarrow \point$ of the blowup of $\mathbb{A}^2$ at the origin $\point = \Spec \CC$ from Example \ref{extwoblowupsofpoint} and its normal cone $\Cl{D/\point} = \mathbb{P}^1$. 

For the sake of contradiction, let $\widhat{X} = \mathbb{P}^1$ and $X = Y = \point$ as in the theorem. Endow $\Cl{\point/\point} = \point$ with the initial \lpot, $E = \point$. Then 

\[\lvfc{\widhat{X}}{E} = \lvfc{D}{E} = [\mathbb{P}^1]\]
{\noindent}and
\[\lvfc{X}{E} = \lvfc{\point}{E} = [\point],\]

{\noindent}but again $p_*[\mathbb{P}^1] = 0$ for dimension reasons. 

\end{example}

\begin{theorem}[Commutativity of Log Gysin Map]\label{commlogpb}
Given a composable pair of DM-type maps between log algebraic stacks
\[X \overset{f}{\rightarrow} Y \overset{g}{\rightarrow} Z,\]
outfit $f$, $g$, and $g \circ f$ with log obstruction theories $F$, $G$, $E$ and a compatibility datum (Remark \ref{rmk:cptblitydatumextendobstthies}). Require $X$ to admit stratifications by global quotients.

If $k : V \rightarrow Z$ is a log smooth and equidimensional $Z$-stack and $k$ is DM-type, take fs pullbacks: 
\[\begin{tikzcd}
T \arrow[r] \arrow[d] \lpb       &U \arrow[r] \arrow[d] \lpb      &V \arrow[d]      \\
X \arrow[r]       &Y \arrow[r]      &Z.
\end{tikzcd}\]

Then the equality
\begin{equation}\label{equalityofconesfortriangle}
[\Cl{g \circ f} \subseteq E] = [\Cl{\Cl{g}|_X/\Cl{g}} \subseteq F \oplus G|_X]
\end{equation}
holds on $X$.

\end{theorem}

\begin{proof}

Pullback via $k$ all obstruction theories and their compatibility datum to reduce to showing the theorem for $k: V \longequals Z$. We essentially apply \cite[Theorem 4.8]{virtpb} to $X \rightarrow \Log Y \rightarrow \Log^2 Z$, endowed with the compatible triple $F, G, E$ by composing with an isomorphism of distinguished triangles: 
\[\begin{tikzcd}
G|_X \arrow[r] \arrow[d]        &F \arrow[r] \arrow[d]      &E \arrow[d]      \\
\ccx{\Log Y/\Log Z}|_X \arrow[r] \arrow[d, "\sim"]      &\lccx{X/Z} \arrow[r] \arrow[d, equals]     &\lccx{X/Y}\arrow[d, "\sim"]        \\
\ccx{\Log Y/\Log^2Z}|_X \arrow[r]     &\ccx{X/\Log Z} \arrow[r]        &\ccx{X/\Log^2 Z}. 
\end{tikzcd}\]

Use Lemma \ref{ftfactor} repeatedly to obtain a strict diagram with $U, V$ quasicompact and \'etale over the stacks $\Log Y, \Log^2 Z$:

\[\begin{tikzcd}
X \arrow[r] \arrow[dr, bend right]       &U \arrow[r] \arrow[d]      &V \arrow[d]      \\
        &\Log Y \arrow[r]     &\Log^2 Z.
\end{tikzcd}\]

Endow the cone $\C{\Log Y/\Log Z}$ with the pullback log structure from $\Log Y$ and pull it back along the part of the diagram above $\Log Y$:
\[\begin{tikzcd}
\C{\Log Y/\Log Z}|_X=\C{U/V}|_X \arrow[r] \arrow[dr, bend right]     &\C{U/V} \arrow[d]        \\
        &\C{\Log Y/\Log Z}.
\end{tikzcd}\]
The triangle is strict and the map $\C{U/V} \to \C{\Log Y/\Log Z}$ is pulled back from the \'etale $U \to \Log Y$, so
\[\Cl{\C{\Log Y/\Log Z}|_X/\C{\Log Y/\Log Z}} = \C{\C{U/V}|_X/\C{U/V}}.\]

Write $i : X \to U$ $j : U \to V$ for the maps. Then the compatibility datum pulls back and \cite[Theorem 4.8]{virtpb} gives us
\[(j \circ i)^!_E ([V]) = i_F^! \circ j_G^!([V]).\]

Unwinding definitions, this becomes 
\begin{equation}\label{openstrictversionequalityofconesfortriangle}
[\C{X/V} \subseteq E] = [\C{\C{U/V}|_X/\C{U/V}} \subseteq F \oplus G|_X].
\end{equation}

This may be rewritten as 
\[[\Cl{X/Z} \subseteq E] = [\Clstrict{\C{\Log Y/\Log Z}|_X/\C{\Log Y/\Log Z}} \subseteq F \oplus G|_X],\]
the claimed equality of classes.

\end{proof}

\begin{remark}
Theorem \ref{commlogpb} says that 
\[(g \circ f)^! = f^! g^!\]
in the sense that any log smooth, equidimensional log stack over $Z$ has rationally equivalent images under these two operations. 

\end{remark}

\begin{remark}\label{rmk:constructlpotandcptibilitydatumforsquare}

Consider an fs pullback of DM-type morphisms between log algebraic stacks: 
\[\begin{tikzcd}
X' \ar[r, "p"] \ar[d, "f'", swap] \lpb      &X \ar[d, "f"]      \\
Y' \ar[r, "q"]      &Y.
\end{tikzcd}\]
Write $r : X' \to Y$ for the composite $f \circ p = q \circ f'$. If $f, q$ are endowed with \lpot's $\Cl{X/Y} \subseteq F$, $\Cl{Y'/Y} \subseteq E$, how should we give $r$ a \lpot?

The fs pullback square induces a pullback of stacks, which may be reexpressed as a ``magic square:''
\[\begin{tikzcd}
\Log X' \ar[r] \ar[d] \pb         &\Log X \ar[d]         &\phantom{a} \ar[dr, phantom, "\rightsquigarrow"]       &       &\Log X' \ar[d] \ar[r] \pb        &\Log X \times \Log Y' \ar[d]        \\
\Log Y' \ar[r]         &\Log Y         &       &\phantom{a}       &\Log Y \ar[r]         &\Log Y \times \Log Y.
\end{tikzcd}\]
The magic square induces a closed immersion 
\[\C{\Log X'/\Log Y} \subseteq \C{\Log X /\Log Y}|_{\Log X'} \times_{\Log X'} \C{\Log Y'/\Log Y}|_{\Log X'}\]
which pulls back to a closed immersion
\[\Cl{X'/Y} \subseteq \C{\Log X /\Log Y}|_{X'} \times_{X'} \C{\Log Y'/\Log Y}|_{X'}\]
on $X'$. As in Remark \ref{rmk:cptblitydatumextendobstthies}, we have closed embeddings $\C{\Log X/\Log Y} \subseteq \Cl{X/Y}|_{\Log X}$, $\C{\Log Y'/\Log Y} \subseteq \Cl{Y'/Y}|_{\Log Y'}$. We endow $r$ with the \lpot given by the composite:
\[\Cl{X'/Y} \subseteq \C{\Log X /\Log Y}|_{X'} \times_{X'} \C{\Log Y'/\Log Y}|_{X'} \subseteq \Cl{X/Y}|_{X'} \times_{X'} \Cl{Y'/Y}|_{X'} \subseteq F|_{X'} \times_{X'} E|_{X'}.\]

We now construct a compatibility datum for the triangle $r = q \circ f'$, leaving the reader to apply the same argument to the other triangle $r = f \circ p$. By the definitions of the \lpot's, we have a commutative diagram: 
\[\begin{tikzcd}
\Cl{X'/Y'} \ar[d, hook] \ar[r]      &\Cl{X'/Y} \ar[r] \ar[d, hook]      &\C{\Log Y'/\Log Y}|_{X'} \ar[d, hook]       \\
F|_{X'} \ar[r, "(0 \times id)"]       &E|_{X'} \times_{X'} F|_{X'} \ar[r]         &E|_{X'}.
\end{tikzcd}\]
To be clear, the morphism $F|_{X'} \to E|_{X'} \times_{X'} F|_{X'}$ is the vertex map times the identity. It's clear the bottom row comes from a distinguished triangle in the derived category and the top row comes from Remark \ref{rmk:natltyofsessandotherdisttriangle}. 

\end{remark}

\begin{corollary}
Suppose given an fs pullback square 
\[\begin{tikzcd}
X' \ar[r, "p"] \ar[d, "f'", swap] \lpb      &X \ar[d, "f"]      \\
Y' \ar[r, "q"]      &Y
\end{tikzcd}\]
of DM-type morphisms between log algebraic stacks which admit stratifications by quotient stacks. Outfit $q$ with a \lpot $E$ and $f$ with a \lpot $F$; give $p, f'$ the pullback obstruction theories. Then 
\[f'^! \circ q^! = p^! \circ f^!\]
in the sense that the operations send any log smooth equidimensional input stack to the same class in $A_* X'$. 

\end{corollary}

\begin{proof}

Denote by $r : X' \to Y$ the map $f \circ p = q \circ f'$. Apply Theorem \ref{commlogpb} to both commutative triangles using the compatibility datum constructed in Remark \ref{rmk:constructlpotandcptibilitydatumforsquare} to see that 
\[p^! \circ f^! = r^! = f'^! \circ q^!.\]

\end{proof}

\section{The Log Costello Formula}

This section proves a log analogue of the Costello Formula \cite[Theorem 5.0.1]{costello}. We will have more to say building on future work \cite{ourcorrectiontocostello}.

\begin{theorem}\label{logcostello2}
Consider an fs pullback square of DM-type maps between algebraic stacks: 
\[\begin{tikzcd}
X' \arrow[r, "p"] \arrow[d, "f'"] \lpb      &X \arrow[d, "f"]       \\
Y' \arrow[r, "q"]        &Y.
\end{tikzcd}\]

Assume
\begin{itemize}
    \item $Y' \to Y$ is of some pure degree $d \in \mathbb{Q}$ as in \cite[Theorem 5.0.1]{costello},
    \item $Y', Y$ are both log smooth and equidimensional,
    \item all arrows are DM-type and all stacks are locally noetherian and locally finite type over $\CC$,
    \item $X', X$ admit stratifications by global quotient stacks \cite{kreschthesis}
    \item $q$ is proper. 
\end{itemize}

Endow $f$ with a log perfect obstruction theory $E$ and give $f'$ the pullback obstruction theory. Then 
\[p_*\lvfc{X'}{E|_{X'}} = d \cdot \lvfc{X}{E}\]
in the Chow Ring of $X$. 

\end{theorem}

\begin{remark}\label{puredegreeafterthekatofactorization}
Let $Y' \rightarrow Y$ be a map between log smooth, equidimensional stacks which is of pure degree $d$. Let $W \rightarrow Y$ be a smooth, log smooth, integral, and saturated morphism and $\widtilde{W} \rightarrow W$ a log blowup. Form the fs pullback diagram: 
\[\begin{tikzcd}
\widtilde{W'} \ar[r] \ar[d] \lpb       &\widtilde{W}\ar[d]       \\
W' \ar[r] \ar[d] \lpbstrict      &W \ar[d]       \\
Y' \ar[r]      &Y.
\end{tikzcd}\]

The property ``of pure degree $d$'' pulls back along smooth morphisms, so it applies to $W' \to W$. Then \cite[Proposition 4.3]{niziol} shows that $\widtilde{W} \to W$ is birational, so $\widtilde{W'} \to \widtilde{W}$ is also of pure degree $d$. 

\end{remark}

\begin{proof}[Proof of Theorem \ref{logcostello2}]

Consider the morphism  
\[s : \Cl{X'/Y'} \rightarrow \Cl{X/Y}.\]

We will prove that $s$ is of pure degree $d$. Both ``of pure degree'' and the specific degree $d$ can be checked after pulling back $s$ along a strict, smooth cover of $\Cl{X/Y}$. Lemmas \ref{lem:strsmloctargsmloccone}, \ref{lem:strsmlocsrcsmloccone} show that replacing $Y$ or $X$ by a smooth cover results in such a smooth cover of cones.

We may thereby assume $X$ and $Y$ are log schemes and the map $f$ globally factors as in Construction \ref{katofactorization}: 
\[X \rightarrow X_\theta \rightarrow \mathbb{A}^{r+s}_Y \to Y.\]

Note $\mathbb{A}^{r+s}_Y \rightarrow Y$ is smooth, log smooth, integral, and saturated, and $X_\theta \rightarrow \mathbb{A}_Y^{r+s}$ is a log blowup. We are in the situation of Remark \ref{puredegreeafterthekatofactorization}, so pulling back:
\[\begin{tikzcd}
X' \ar[r] \ar[d, hook] \lpbstrict      &X \ar[d, hook]      \\
X'_\theta \ar[r] \ar[d] \lpb       &X_\theta \ar[d]        \\
Y' \ar[r]      &Y
\end{tikzcd}\]
results in a map $X'_\theta \rightarrow X_\theta$ which is pure of degree $d$ along $X \rightarrow X_\theta$. The proof of Costello's Formula \cite[Theorem 5.0.1]{costello} then asserts that 
\[\Clstrict{X'/X'_\theta} \rightarrow \Clstrict{X/X_\theta}\]
is of pure degree $d$. The short exact sequences of Proposition \ref{targetsmdescentnormalcone}
\[\begin{tikzcd}
\Tl{X'_\theta/Y'} \ar[r] \ar[d]      &\Clstrict{X'/X'_\theta} \ar[r] \ar[d, "t"]     &\Cl{X'/Y'} \ar[d, "s"]     \\
\Tl{X_\theta/Y} \ar[r]       &\Clstrict{X/X_\theta} \ar[r]     &\Cl{X/Y}     
\end{tikzcd}\]
let us conclude that $s$ is as well. 

\end{proof}

\section{The Product Formula}

Let $V$, $W$ be log smooth, quasiprojective schemes throughout this section. We denote the stacks of \textit{prestable curves} and \textit{stable curves} which have $n$-markings and genus $g$ by $\Mprel, \Ms$, respectively \cite[0DMG]{sta}. They are endowed with divisorial log structures coming from the locus of singular curves \cite[1.5, Appendix A]{loggw}, \cite{fkatomodulilogcurves}.

\begin{definition}[Log Stable Maps]
The stack of log stable maps $\Ml(V)$ has fiber over an fs log scheme $T$ the category of diagrams of fs log schemes 
\[\begin{tikzcd}
C \arrow[r] \arrow[d]      &V      \\
T
\end{tikzcd}\]
with $C \to T$ a log smooth curve \cite[Definition 1.2]{fkatomodulilogcurves} of genus $g$ and $n$ marked points, such that the underlying diagram of schemes is a stable map of curves. 

\end{definition}

Remarkably, the log algebraic stack $\Ml(\Spec \CC)$ of log curves without a map is isomorphic to the ordinary stack of stable curves $\Ms$ with log structure induced by the boundary of degenerate curves \cite[Theorem 4.5]{fkatomodulilogcurves}. The log structures of $\Ml(V)$ for a general fs target may be more complicated, as they have to do with the ``tropical deformation space'' of the curve \cite{loggw}. 

\begin{comment}
do an example of the one smoothing of the node influencing the other
\end{comment}

\begin{construction}[{\cite[Section 5]{loggw}}]\label{lpotforlogstablemaps}
We recall the construction \cite[Section 5]{loggw} of the natural \lpot\, for $\Ml(V) \rightarrow \Mprel$ to clarify differences in notation. 

Write $\UU \to \Mprel$ for the universal curve. Define $\UU_V$ as the fs pullback, naturally equipped with a tautological map to $V$:
\[\begin{tikzcd}
V       &\UU_V \arrow[r, "\pi_V"] \arrow[d] \arrow[l] \lpbstrict      &\Ml(V) \arrow[d]         \\
        &\UU \arrow[r]        &\Mprel
\end{tikzcd}\]
This diagram induces maps between log cotangent complexes
\[\lccx{V}|_{\UU_V} {\longrightarrow} \lccx{\UU_V/\UU}   \overset{t}{\longleftarrow}  \lccx{\Ml(V)/\Mprel}|_{\UU_V}.\]
The map $\UU \to \Mprel$ is integral, saturated, and log smooth according to its functor of points, so its underlying map of stacks is flat and the fs pullback square is also an ordinary pullback. 

Then $t$ is an isomorphism \cite[1.1 (iv)]{logcotangent}, and the log cotangent complex of $V$ is \cite[1.1 (iii)]{logcotangent}
\[\lccx{V} = \lkah{V}[0].\]
We've written $[0]$ to consider a coherent sheaf as a chain complex concentrated in degree $0$. Via the isomorphism $t$ and this identification, we have obtained a map
\begin{equation}\label{eqn:adjointtoobstthy}
\lkah{V}[0]|_{\UU_V} \to \lccx{\UU/\Mprel}|_{\UU_V}.\end{equation}

We need the ordinary relative dualizing sheaf $\omega_{\pi_V^\circ}$ and the identification 
\[L \pi_V^! (\cdot) = \omega_{\pi_V^\circ} \overset{L}{\otimes} L \pi_V^* (\cdot). \]

Tensor \eqref{eqn:adjointtoobstthy} by $\omega_{\pi_V^\circ}$ and use the adjunction:
    \[\Omega^\ell_{V}[0]\,|_{\UU_V} \overset{L}{\otimes} \omega_{\pi_V^\circ} \longrightarrow L \pi_V^! \lccx{\Ml(V)/\Mprel},\]     
    \[E(V) := R\pi_{V *}(\Omega^\ell_{V}[0]\,|_{\UU_V} \overset{L}{\otimes} \omega_{\pi_V^\circ}) \longrightarrow \lccx{\Ml(V)/\Mprel}.\]
We won't repeat the verification \cite[Proposition 5.1]{loggw} that $E(V)$ is a \lpot. 

\end{construction}

\begin{remark}

The map \eqref{eqn:adjointtoobstthy} comes from the map on normal cones
\[\Cl{\Ml(V)/\Mprel}|_{\UU_V} \overset{\sim}{\longleftarrow} \Cl{\UU_V/\UU} \longrightarrow B\Tl{V}|_{\UU_V}.\]
We needed duality, so we opted for the other perspective. 

\end{remark}

\begin{remark}[Variants]
The reader may choose to work in the relative setting of a log smooth and quasiprojective map $V \to S$. Obstruction Theories are obtained in the same way. 

We can naturally impose ``contact order'' conditions \cite{wisebounded} in the log setting, but we only fix genus and number of markings to be consistent with \cite{logprodfmla}. The reader may readily vary the numerical type conditions in our formulas. 
\end{remark}

We need one more stack, $\mathfrak{D}$: Points of $\mathfrak{D}$ over $T$ are diagrams $(C' \leftarrow C \rightarrow C'')$ of genus $g$, $n$-pointed prestable curves over $T$ whose maps are partial stabilizations (they lie over the identities in $\Ms$) that don't both contract any component. In other words, $C \to C' \times C''$ itself is a stable map. This stack is only necessary to form an fs pullback square:

\begin{situation}[{\cite[Section 2]{logprodfmla}}]\label{sit:thechaseddiagram}
Recall the fs pullback square: 
\begin{equation}\label{notyetfactoredsquare}
\begin{tikzcd}
\Ml(V \times W) \arrow[r] \arrow[d, "c"] \lpb        &\Ml(V) \times \Ml(W) \arrow[d, "a"]     \\
\mathfrak{D} \arrow[r, "\widtilde{\Delta}"]        &\Mprel \times \Mprel
\end{tikzcd}\end{equation}

Let $C \to V \times W$ be a log stable map over a base $T$. The maps $(C \to V)$, $(C \to W)$ needn't be stable; denote their stabilizations by $(C' \to V)$, $(C'' \to W)$, respectively. 

The top horizontal arrow in \eqref{notyetfactoredsquare} sends $(C \to V \times W)$ to the induced log stable maps $(C' \to V, C'' \to W)$. The vertical arrow $c$ sends $(C \to V \times W)$ to the partial stabilizations $(C' \leftarrow C \to C'')$. The map $\widtilde{\Delta}$ sends a diagram $(C' \leftarrow C \to C'')$ to the pair of prestable curves $C', C''$. Finally, $a$ sends a pair of log stable maps $(C' \to V, C'' \to W)$ to the prestable curves $(C', C'')$.

This square has a factorization:
\begin{equation}\label{completelyfactoredsquare}
\begin{tikzcd}
\Ml(V \times W) \arrow[r, "h"] \arrow[d, "c"] \lpb        &Q \arrow[r] \arrow[d] \lpb     &\Ml(V) \times \Ml(W) \arrow[d, "a"]     \\
\mathfrak{D} \arrow[r, "l"]        &Q' \arrow[r, "\phi"] \arrow[d] \lpb       &\Mprel \times \Mprel \arrow[d, "s \times s"]      \\
        &\Ms \arrow[r, "\Delta"]        &\Ms \times \Ms,
\end{tikzcd}\end{equation}
where $s: \Mprel \to \Ms$ stabilizes a prestable curve.

To be clear, $Q = \Ml(V) \times_{\Ms}^\ell \Ml(W)$ and $Q' = \Mprel \times_{\Ms}^\ell \Mprel$ are the analogues of \cite{logprodfmla}'s $P$, $\mathfrak{P}$, etc. 
\end{situation}

\begin{theorem}[The ``Log Gromov-Witten Product Formula'']\label{loggwfmla}
With $V$, $W$ log smooth, quasiprojective schemes, 
\[h_*\lvfc{\Ml(V \times W)}{E(V \times W)} = \Delta^!(\lvfc{\Ml(V)}{E(V)} \times \lvfc{\Ml(W)}{E(W)}).\]
\end{theorem}

Our proof will be the same as K. Behrend's \cite{prodfmla}: we compute the log normal cone of the map $Q \rightarrow Q'$ in two different ways.

\begin{remark}[{On Diagram \eqref{completelyfactoredsquare}}]
We equip $a$ with the product $E(V) \boxplus E(W)$ of the natural \lpot's of Construction \ref{lpotforlogstablemaps}, adopting the notation
\[E \boxplus E' := E|_{V \times W} \oplus E'|_{V \times W}.\]

The cotangent complex $\lccx{\Delta}$ is of perfect amplitude in [-1, 0] because its source and target are log smooth. Therefore $\Cl{\Delta} = \Nl{\Delta}$ serves as a natural \lpot for itself. We equip $\phi$ with the pullback obstruction theory, resulting in 
\[\Delta^! = \phi^!\]
by Remark \ref{gysinmapofpullbackisthesame}. We endow the square bounded by $\phi$ and $a$ with the natural compatibility datum afforded all such squares as in Remark \ref{rmk:constructlpotandcptibilitydatumforsquare}.

All of the arrows in Diagrams \eqref{notyetfactoredsquare} and \eqref{completelyfactoredsquare} are of DM-type. 
\end{remark}

\begin{lemma}
The stabilization map $s : \Mprel \to \Ms$ is log smooth. 
\end{lemma}

\begin{proof}

The cover $\bigsqcup_m \Msp{m} \to \Mprel$ given by forgetting marked points and not stabilizing is strict smooth \cite[1.2.1]{logprodfmla}. This map is in particular kummer and surjective, and \cite[Theorem 0.2]{logflatdescent} applies with $\mathbb{P}=$ ``log smooth'' once we argue that the composite $\bigsqcup_m \Msp{m} \to \Ms$ is log smooth. 

The forgetful map $\Msp{1} \to \Ms$ is the universal curve, so it is tautologically log smooth. We see the map $\Msp{m} \to \Ms$ is log smooth by iterating this forgetfulness, and this completes the argument.

\end{proof}

\begin{remark}

The map $\mathfrak{D} \to \Mprel$ which records the initial curve is log \'etale since the original map was \'etale \cite[Lemma 4]{prodfmla} and ours is the fsification thereof. The stack $Q'$ is log smooth because the map $Q' \to \Ms$ is pulled back from $s \times s$. 

Given a log \'etale map $X' \to X$ of log smooth log algebraic stacks with $X$ equidimensional, we claim $X'$ must be as well. The maps $X' \subseteq \Log X'$, $X \subseteq \Log X$ are dense because of the log smoothness assumption and the map $\Log X' \to \Log X$ is \'etale. Thus $\Log X$ and $\Log X'$ are equidimensional, as well as $X' \subseteq \Log X'$. This argument shows that fsification preserves equidimensionality of log smooth stacks, so our fs versions of $\mathfrak{D}$, $Q'$ are equidimensional because the original versions \cite{prodfmla} were.

\end{remark}

\begin{lemma}
The obstruction theories $E(V)$, $E(W)$, $E(V \times W)$ are compatible in the sense that
\[\widtilde{\Delta}^*(E(V) \boxplus E(W)) \simeq E(V \times W).\]
\end{lemma}

\begin{proof}

We completely echo the proof of \cite[Proposition 6]{prodfmla}.

Consider the diagram of universal log curves and tautological maps with the notation:
\[\begin{tikzcd}
V       &       &V\times W  \arrow[ll] \\
\UU_V \arrow[u, "f_V"] \arrow[d, "\pi_V", swap]       &\UUU_V \arrow[l, "s_V", swap]  \arrow[dr, "\widtilde{\pi}_V", swap] \arrow[dl, very near start, phantom, "\msout{\ell} \urcorner"]    &\UU_{V \times W}  \arrow[l, "q_V", swap] \arrow[u, "f_{V \times W}", swap] \arrow[d, "\pi_{V \times W}"]             \\
\Ml(V)      &       &\Ml(V \times W)  \arrow[ll, "r_V"].         \\
\end{tikzcd}\]

We claim $F \rightarrow Rq_{V*}q_V^*F$ is an isomorphism for any vector bundle $F$ on $\UU_V$. The map $q_V$ represents partial stabilization. We make the argument for contracting one $\mathbb{P}^1$ at a time. 

We first compute that $R^pq_{V*}q_V^*F = 0$ for $p \neq 0$. This claim is local in $\UU_V$, so assume $F$ is trivial. The fiber of $R^p q_{V*}q_V^*F$ at a point $x$ is $H^p(q_V^{-1}(x), q_V^*F)$. Hence the fibers $q_V^{-1}(x)$ are either a point or $\mathbb{P}^1$. On each fiber, the cohomology of the trivial vector bundle is concentrated in degree 0 \cite[01XS]{sta}. Not only are $F$ and $q_{V*}q_V^*F$ abstractly isomorphic in that case, but the natural map is an isomorphism \cite[Exercise 9.3.11]{fgaexplained}.

The universal curve $\pi_V$ is tautologically flat, integral, and saturated. The fs pullback square it belongs to is therefore also an ordinary flat pullback, subject to cohomology and base change \cite[Tag 08IB]{sta}. This gives:
\begin{align*}
    L r_V^* R\pi_{V*} Lf_V^* \Omega_V  &= R \widtilde{\pi}_{V*} L s_V^* Lf_V^* \Omega_V     \\
        &= R \widtilde{\pi}_{V*} Rq_{V*} q_V^* L s_V^* Lf_V^* \Omega_V        \\
        &= R\pi_{V \times W*} Lf^*_{V \times W} (\Omega_V|_{V \times W}).
\end{align*}
All the same goes for $W$. Add the two together to get 
\[L r_V^* R\pi_{V*} Lf_V^* \Omega_V \boxplus L r_W^* R\pi_{W*} Lf_W^* \Omega_W\,\, = \,\, R\pi_{V \times W*} Lf^*_{V \times W} (\Omega_V \boxplus \Omega_W).\]

This is dual to the compatibility we set out to prove, so we are through. 

\end{proof}

\begin{proof}[Proof of Theorem \ref{loggwfmla}]

Compute the log virtual fundamental class $\vfc{Q}{E(V) \boxplus E(W)}$ in two different ways: 
\begin{align*}
    \vfc{Q}{E(V) \boxplus E(W)} &:= [\Cl{Q/Q'} \subseteq E(V) \boxplus E(W)]        \\
            &=a^!(Q')        \\
            &=a^! \phi^! (\Mprel \times \Mprel)        \\
            &=\phi^! a^!  (\Mprel \times \Mprel)       \\
            &=\Delta^! \vfc{\Ml(V) \times \Ml(W)}{E(V) \boxplus E(W)}.          \intertext{On the other hand,}
\\
    \vfc{Q}{E(V) \boxplus E(W)}     &= h_* \vfc{\Ml(V \times W)}{E(V \times W)}
\end{align*}
by the Log Costello Formula \ref{logcostello2}. 

\end{proof}

\bibliographystyle{alpha}
%Used BibTeX style unsrt is in order of mention, alpha is alphabetical, far better
\bibliography{zbib}

\end{document}